\documentclass[10pt, a4paper, reqno]{amsart}

\usepackage[left=1.1in,right=1.1in,top=1.1in,bottom=1.1in]{geometry} 
\usepackage[utf8]{inputenc}
\usepackage{enumitem}
\usepackage{amsthm}
\usepackage{amssymb}
\usepackage{amsmath}
\usepackage{amscd}
\usepackage{mathtools}
\usepackage{graphicx}
\usepackage{float}
\usepackage{cite}
\usepackage{xcolor}
\usepackage{listings}
\usepackage[most]{tcolorbox}
\tcbuselibrary{listings}
\usepackage{tikz-cd}
\usepackage{tikz,tikz-cd}
\lstdefinestyle{lean}{
  basicstyle=\ttfamily\small,
  keywordstyle=\bfseries,
  morekeywords={import,example,by,have,rw,exact},
  columns=[l]fixed,
  keepspaces=true,
  showstringspaces=false,
  breaklines=true,
  breakatwhitespace=true,
  frame=none,
  tabsize=2,
  escapeinside={(*@}{@*)},
  literate={
    {ℕ}{{\ensuremath{\mathbb{N}}}}1
    {ℝ}{{\ensuremath{\mathbb{R}}}}1
    {∃}{{\ensuremath{\exists}}}1
    {∀}{{\ensuremath{\forall}}}1
    {≤}{{\ensuremath{\leq}}}1
    {≥}{{\ensuremath{\geq}}}1
    {∧}{{\ensuremath{\wedge}}}1
    {∨}{{\ensuremath{\vee}}}1
    {∞}{{\ensuremath{\infty}}}1
    {→}{{\ensuremath{\to}}}1
    {←}{{\ensuremath{\leftarrow}}}1
    {↔}{{\ensuremath{\leftrightarrow}}}1
    {‖}{{\ensuremath{\Vert}}}1
    {⁻¹}{{\ensuremath{{}^{-1}}}}2
    {∑}{{\ensuremath{\textstyle\sum}}}1
    {∑'}{{\ensuremath{\textstyle\sum\nolimits^{\prime}}}}2
    {∈}{{\ensuremath{\in}}}1
    {∉}{{\ensuremath{\notin}}}1
    {⊆}{{\ensuremath{\subseteq}}}1
    {⊂}{{\ensuremath{\subset}}}1
    {⊤}{{\ensuremath{\top}}}1
    {⊥}{{\ensuremath{\bot}}}1
    {⬝ᵥ}{{\ensuremath{\cdot}}}1
    {▸}{{\ensuremath{\triangleright}}}1
  },
}
\usepackage[british]{babel}
\usepackage{esint}
\usepackage{svg}
\usepackage{times}
\usepackage{comment}
\usepackage{hyperref}

\numberwithin{equation}{section}

\renewcommand{\geq}{\geqslant}
\renewcommand{\leq}{\leqslant}

\newcommand{\R}{{\mathbb{R}}}
\newcommand{\Z}{{\mathbb{Z}}}

\newcommand{\SCI}{\operatorname{SCI}}

\def\R{\mathbb{R}}

\def\Z{\mathbb{Z}}

\newtheorem{theorem}{Theorem}[section]

\newtheorem{corollary}[theorem]{Corollary}

\theoremstyle{definition}
\newtheorem{definition}[theorem]{Definition}
\newtheorem{remark}[theorem]{Remark}
\newtheorem*{remark*}{Remark}

\newtheorem{example}[theorem]{Example}

\theoremstyle{definition}
\newtheorem*{notation}{Notation}

\usepackage{xparse}
\newcounter{constant}
\renewcommand{\theconstant}{\arabic{constant}}

\newtcblisting{lstlistinglean}{
  listing only,
  listing engine=listings,
  listing options={
    style=lean,
    basicstyle=\ttfamily\fontsize{7.34pt}{8.8pt}\selectfont,
  },
  colback=black!6,
  boxrule=0pt,
  arc=0pt,
  left=5pt, right=5pt,
  top=2pt, bottom=2pt,
  width=\linewidth,
  enhanced,
}
{}
\newtcblisting{lstlistingleansmall}{
  listing only,
  listing engine=listings,
  listing options={
    style=lean,
    basicstyle=\ttfamily\fontsize{6.8pt}{7.2pt}\selectfont,
  },
  colback=black!6,
  boxrule=0pt,
  arc=0pt,
  left=5pt, right=5pt,
  top=2pt, bottom=2pt,
  width=\linewidth,
  enhanced,
}

\NewDocumentCommand{\declareconstant}{m o}{
	\refstepcounter{constant}
	\label{const:#1}
	\IfValueTF{#2}
	{\expandafter\gdef\csname constdep@#1\endcsname{#2}}
	{\expandafter\gdef\csname constdep@#1\endcsname{}}
	C^{(\theconstant)}%
	\IfValueT{#2}{_{#2}}
}

\NewDocumentCommand{\constant}{m o}{
	\IfValueTF{#2}
	{
		C^{(\ref{const:#1})}_{#2}
	}
	{
		C^{(\ref{const:#1})}%
		\expandafter\ifx\csname constdep@#1\endcsname\empty
		\else
		_{\csname constdep@#1\endcsname}
		\fi
	}
}
\title[Why Lean verification does not guarantee faithful AI autoformalisation]{Navier--Stokes lost in translation \\ -- \\
Why Lean verification of AI autoformalisation does not guarantee correct natural language proofs}

\author{A. Bastounis} 
\address{Department of Mathematics, King's College London}
\email{alexander.bastounis@kcl.ac.uk}

\author{F. Circelli} 
\address{Department of Applied Mathematics and Theoretical Physics, University of Cambridge}
\email{fac43@cam.ac.uk}

\author{A. C. Hansen*} 
\address{Department of Applied Mathematics and Theoretical Physics, University of Cambridge}
\email{a.hansen@damtp.cam.ac.uk}
\thanks{*Corresponding author: a.hansen@damtp.cam.ac.uk}

\begin{document}


\maketitle

\vspace{-4mm}

\begin{abstract}
Autoformalisation is increasingly used to verify mathematical texts, including those generated by AI, as in OpenAI's announced proof of blow-up of solutions to the Navier-Stokes equations. In this process, an AI system translates the text from a natural language (NL) into a formal language such as Lean. Once this translation is done, the argument expressed in the formal language can easily be mechanically verified. The purpose of this article is to demonstrate why this process may offer no confidence in the original NL argument, owing to the various difficulties in performing the translation semantically faithfully. In particular, we highlight that the problem of resolving ambiguities in mathematical NL text, which is necessary in order to provide semantically faithful translation, is arbitrarily high up in the Solvability Complexity Index (SCI) hierarchy/arithmetical hierarchy (the SCI $= \infty$). Hence, informally, providing semantically faithful AI autoformalisation is harder than any computational problem including the Halting problem (which has SCI $= 1$).  To demonstrate the effect of this result  we provide several examples of AI mistranslations of NL statements and proofs into Lean in practice, resulting in mismatches between NL proofs and their Lean `verifications'. These include OpenAI's announced Navier-Stokes proof. In particular, we show that the formalised Lean proof does not correspond to the NL proof of blow-up of solutions to the Navier-Stokes equations. 
\end{abstract}

\section{Introduction: Two types of Autoformalisation}\label{sec:Intro}

We start by describing the two main types of autoformalisation from a natural language (NL) to a formal language:

\begin{enumerate}[label=(\roman*),leftmargin=16pt]
\item \label{item:TranslationNoSemantic}
\textbf{Translating an alleged proof in an NL to obtain a formal Lean proof.}
Suppose a theorem (the statement, not including a proof) has already been faithfully and correctly translated into Lean. Suppose also the theorem is accompanied by an alleged NL proof.   
The goal is to let a machine (which we will call AI) translate the NL proof into Lean to obtain a formal proof of the theorem. Success is declared when the AI has produced a `translation' which Lean accepts as a valid proof (meaning it compiles), with no use of Lean's `\texttt{sorry}' keyword or extra axioms being permitted. 
\end{enumerate}

\begin{enumerate}[label=(\roman*),leftmargin=16pt,resume]
\item \label{item:TranslationSemantic}
\textbf{Faithful semantic translation of a mathematical text.}
Suppose we are given a mathematical text (e.g. a book or a paper) written in an NL. The goal is to let an AI translate every definition, theorem, proposition, proof etc. in the text to Lean code that compiles, see e.g. the example in \S \ref{sec:meta}. The definitions, statements and proofs must themselves be translated, semantically preserving the
source's mathematical content. For a book, this obligation applies throughout and includes the dependencies between concepts and results.
If the stated goal is to validate its proofs, correspondence to the written
arguments must also be established. In particular, a different, provable assertion would not constitute a semantically faithful translation. This autoformalisation problem is much harder than problem (i) above (see \S \ref{sec:SemanticPreservationHard}).
\end{enumerate}
{\bf Warning:} Meeting the success criterion of \ref{item:TranslationNoSemantic} does not mean that the original NL proof is correct. Indeed, a full Lean compilation of the `translation' does not guarantee a semantically faithful translation. The AI could very easily change the semantics of the translation in order to satisfy this success criterion. For example, the NL proof could be wrong while the `translated' proof is correct (see Example \ref{ex:example1}). Another possibility is that both the NL proof and the formal proof are correct, but completely different (see Example \ref{ex:example2}). A third possibility is that the NL proof provides stronger statements than what the formal proof actually establishes, with (of course) different proofs. The latter happens in OpenAI's announced proof of blow-up of Navier-Stokes equations. In all cases the correctness of the formal proof does not say anything about the correctness of the NL proof.

To answer the question: `Is the NL argument correct?' one needs to consider the second form of autoformalisation \ref{item:TranslationSemantic}. This is, however, in general an extremely difficult problem, as we discuss in \S\ref{sec:SemanticPreservationHard}. In fact, it is strictly harder than the Halting problem.

Because of the phenomenon of mistranslations -- as highlighted in this paper -- the NL proof by OpenAI and other autoformalised Lean proofs should not prima facie be trusted without the same peer review process and scrutiny that other proofs are subjected to. The Lean code merely tells us that the theorem is correct, yet the NL proof with intermediate arguments, lemmas and results may be incorrect or have incorrect arguments. 

{\bf Disclaimer:} We do not make claims about the correctness of OpenAI's NL proof, we only make statements about mistranslations into Lean.

\begin{figure}
\centering
\includegraphics[width=0.7\linewidth]{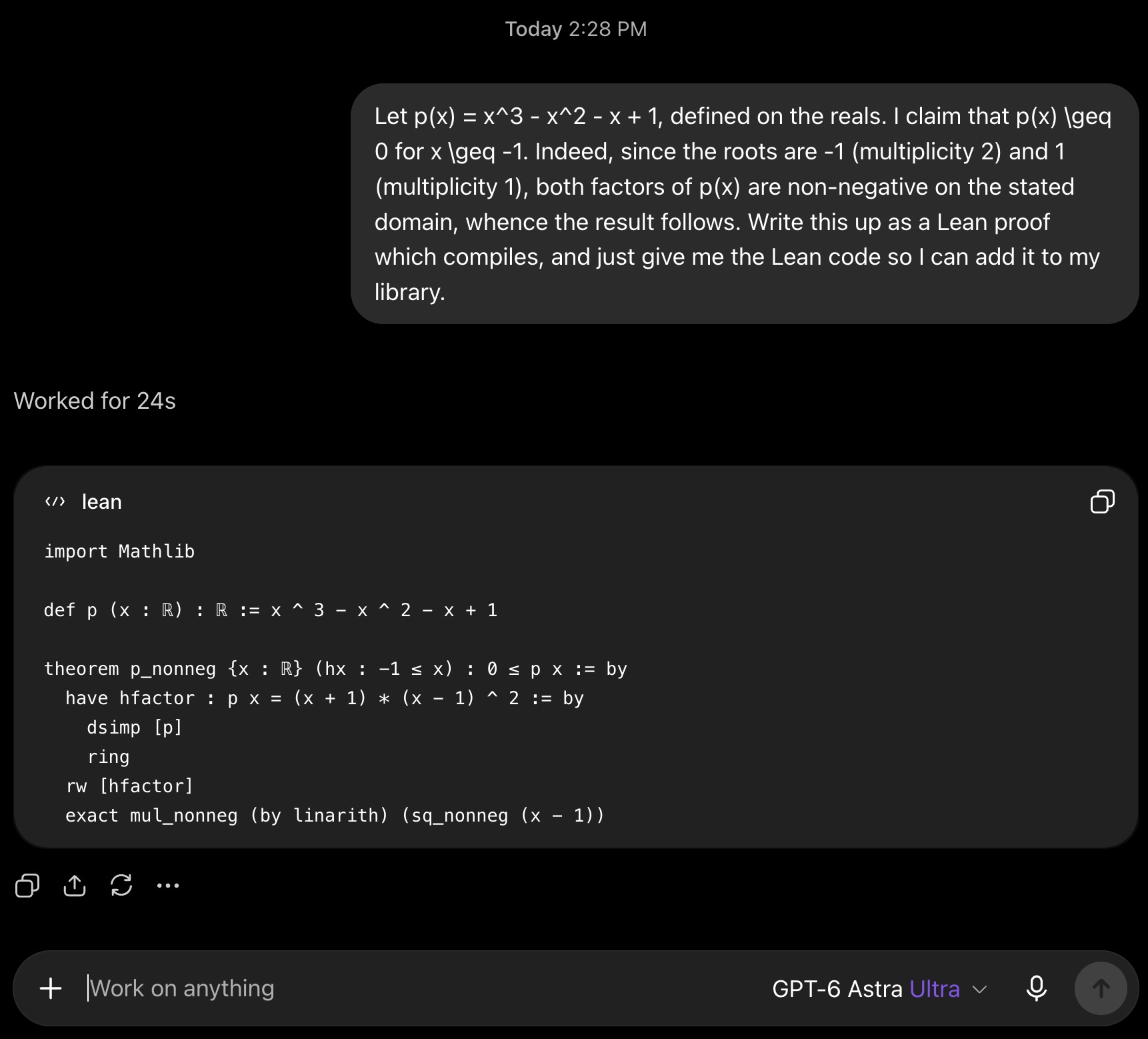}
\caption{ChatGPT-6 (Astra Ultra) `translates' the wrong NL `proof' in Example \ref{ex:example1} into a correct proof in Lean that correctly compiles.}
\label{fig:cubic-chatbot}
\end{figure}

\section{When an AI autoformalised  `translation' is not semantically faithful}\label{sec:first}

In order to illustrate the issues with the first type of autoformalisation described in \S \ref{sec:Intro} we consider several examples of mistranslation. We begin with some elementary examples to highlight the issue, and then turn to more technical examples of mistranslation in OpenAI's announced proof of blow-up of Navier-Stokes equations in \S \ref{sec:OpenAI_NS}.

\begin{example}[A correct statement with an incorrect NL proof, and a `translation' to a different, correct Lean proof]\label{ex:example1} Consider the following theorem:

\begin{center}{\it 
Let \(p(x)=x^3-x^2-x+1\), where $x \in \mathbb{R}$. Then \(p(x)\geq0\) for \(x\geq-1\).
}\end{center}

\begin{proof}[Incorrect NL proof]
The roots are \(-1\), of multiplicity two, and \(1\), of
multiplicity one. Hence\footnote{Note that the `hence' is also dubious without mentioning the leading coefficient of $p$.},
\[
  p(x)=(x-1)(x+1)^2\geq0,
\]
since both factors are nonnegative on the stated domain.
\end{proof}

\noindent The above proof is incorrect; the multiplicities of the stated roots are wrong and hence the claimed factorisation of $p(x)$ is wrong too. 
In the exchange shown in Figure~\ref{fig:cubic-chatbot}, a user presented
the incorrect NL argument above as their `proof' and asked the chatbot to
write this argument as a Lean proof, providing a Lean code that compiles. The chatbot returned a correct Lean proof by finding the correct multiplicities of the roots. 
This example demonstrates how the first form of autoformalisation above can result in a wrong NL proof being mistranslated into a correct Lean proof.
\end{example}

\begin{figure}
\centering
\includegraphics[width=0.7\linewidth]{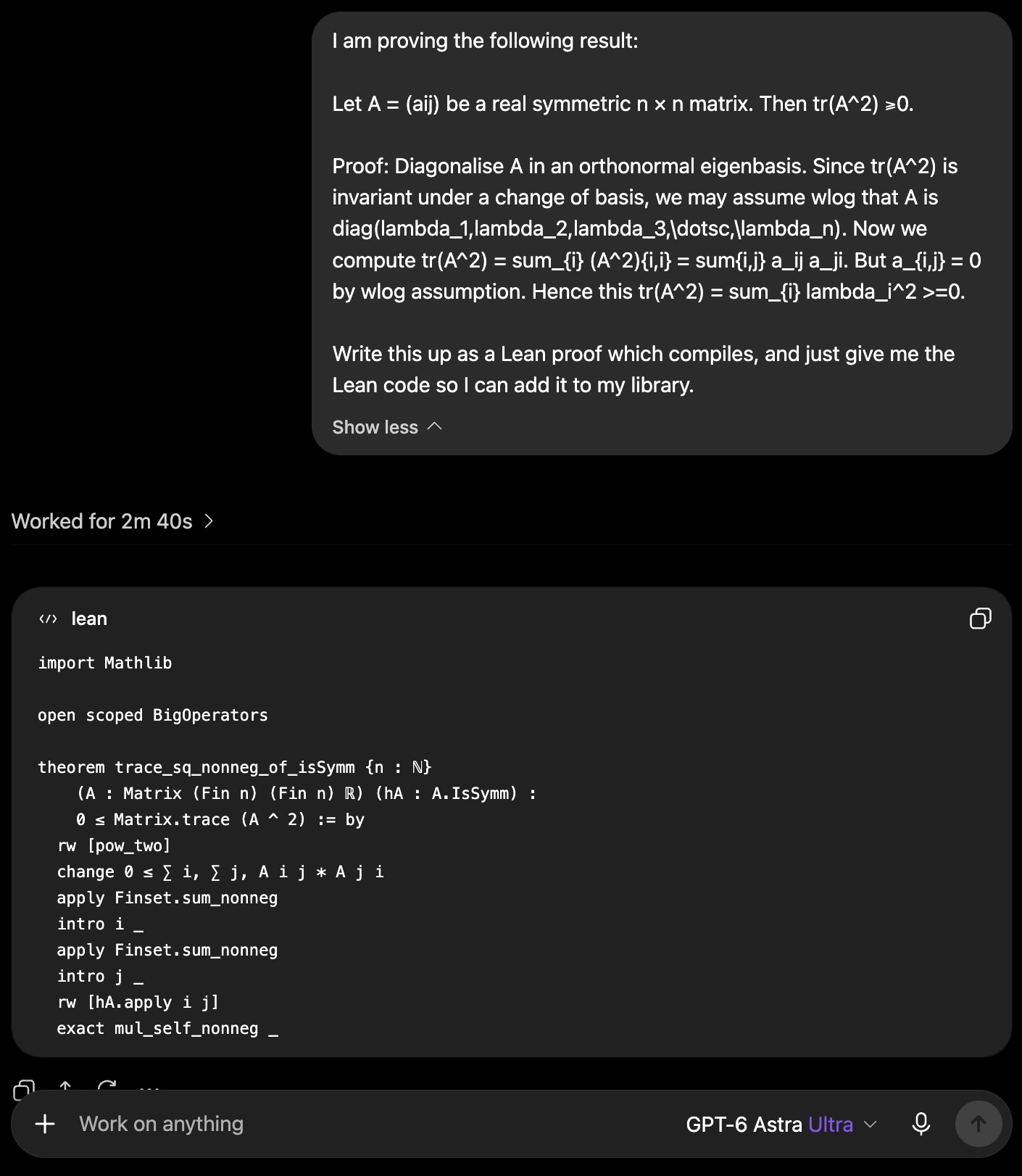}
\caption{ChatGPT-6 (Astra Ultra) `translates' the correct NL proof in Example \ref{ex:example2} into a completely different proof in Lean that compiles.}
\label{fig:matrix-chatbot}
\end{figure}

\begin{example}[A correct statement with a correct NL proof, which is `translated' into a Lean proof with a different mathematical argument ]\label{ex:example2}
Consider the following theorem:

\vspace{2mm}
\begin{center}{\it 
Let \(A = (a_{ij})\) be a real symmetric \(n\times n\) matrix. Then \(\operatorname{tr}(A^2)\geq0\).
}
\end{center}

\begin{proof}[Correct NL proof]
Diagonalise $A$ in an orthonormal eigenbasis. Since
$\operatorname{tr}(A^2)$ is invariant under a change of basis, we may
assume $A=\operatorname{diag}(\lambda_1,\ldots,\lambda_n)$. Now compute
\begin{equation}\label{eq:trace-expansion}
    \operatorname{tr}(A^2)
    =\sum_i(A^2)_{ii}
    =\sum_{i,j}a_{ij}a_{ji}.
\end{equation}
As $a_{ij}=0$ for $i\ne j$ and $a_{ii}=\lambda_i$, this reduces to
$\sum_i\lambda_i^2\geq0$.
\end{proof}

In the exchange shown in Figure~\ref{fig:matrix-chatbot}, a user presented
the correct NL argument above as their proof and asked the chatbot to
write this argument as a Lean proof which compiles. The chatbot, however, returned a different proof displayed in Figure~\ref{fig:matrix-chatbot}. The code gives a correct Lean proof, which has a distinctly different argument compared to the original NL proof that the chatbot was asked to write up in Lean.  Indeed, the Lean proof starts directly at \eqref{eq:trace-expansion}, keeping the original matrix $A$: it then uses $a_{ji}=a_{ij}$ to conclude that
\[
    \sum_{i,j}a_{ij}a_{ji}=\sum_{i,j}a_{ij}^2\geq0.
\]
The unfaithful translation of the argument simply omits the change of basis and then proceeds in the same manner as the NL proof (using symmetry rather than diagonality to obtain a sum of squares). Both proofs are correct, yet this example demonstrates how the first form of autoformalisation above can mistranslate a correct NL proof to yield a completely different valid Lean proof.
\end{example}

\section{Mistranslation in practice -- OpenAI's announced Navier-Stokes proof}\label{sec:OpenAI_NS}
The three-dimensional incompressible Navier-Stokes regularity problem has a long history \cite{CaffarelliKohnNirenberg1982, CordobaMartinezZoroaZheng2026Hypodissipative, Tao2016Averaged, Serrin1962, NecasRuzickaSverak1996, FujitaKato1964, Seregin2012Blowup, Hou2023NavierStokes, CordobaMartinezZoroaZheng2026Hypodissipative} (we can only cite a small subset of papers here), dating back to Leray's fundamental paper from 1934 \cite{Leray1934}, and was designated a Clay Millennium Prize Problem in 2000 \cite{Fefferman2006}. In the GitHub repository accompanying its recently announced proof of finite-time blow-up of solutions to the Navier-Stokes equations, OpenAI writes:
\begin{quote}
	\emph{
		`This repository contains Lean 4 formalizations of the results presented in ``Finite time blowup for Navier--Stokes'' [their announced NL proof] and ``Finite time blowup for the Euler equation'' by OpenAI.'}
\end{quote}
It is natural to ask whether this formalisation in Lean represents a semantically faithful translation of the NL paper, so that we can conclude that the NL paper represents a correct proof? 
Interestingly, OpenAI's announced proof becomes a great source for examples of mistranslations, exactly as pointed out in \S \ref{sec:first}. In particular, the Lean formalisation does not represent a semantically faithful translation of the NL paper, meaning that the Lean formalisation provides no guarantee that the NL proof is correct. We provide several examples below (see also Remark \ref{rem:further}).

\begin{figure}
\centering
\includegraphics[width=0.9\linewidth]{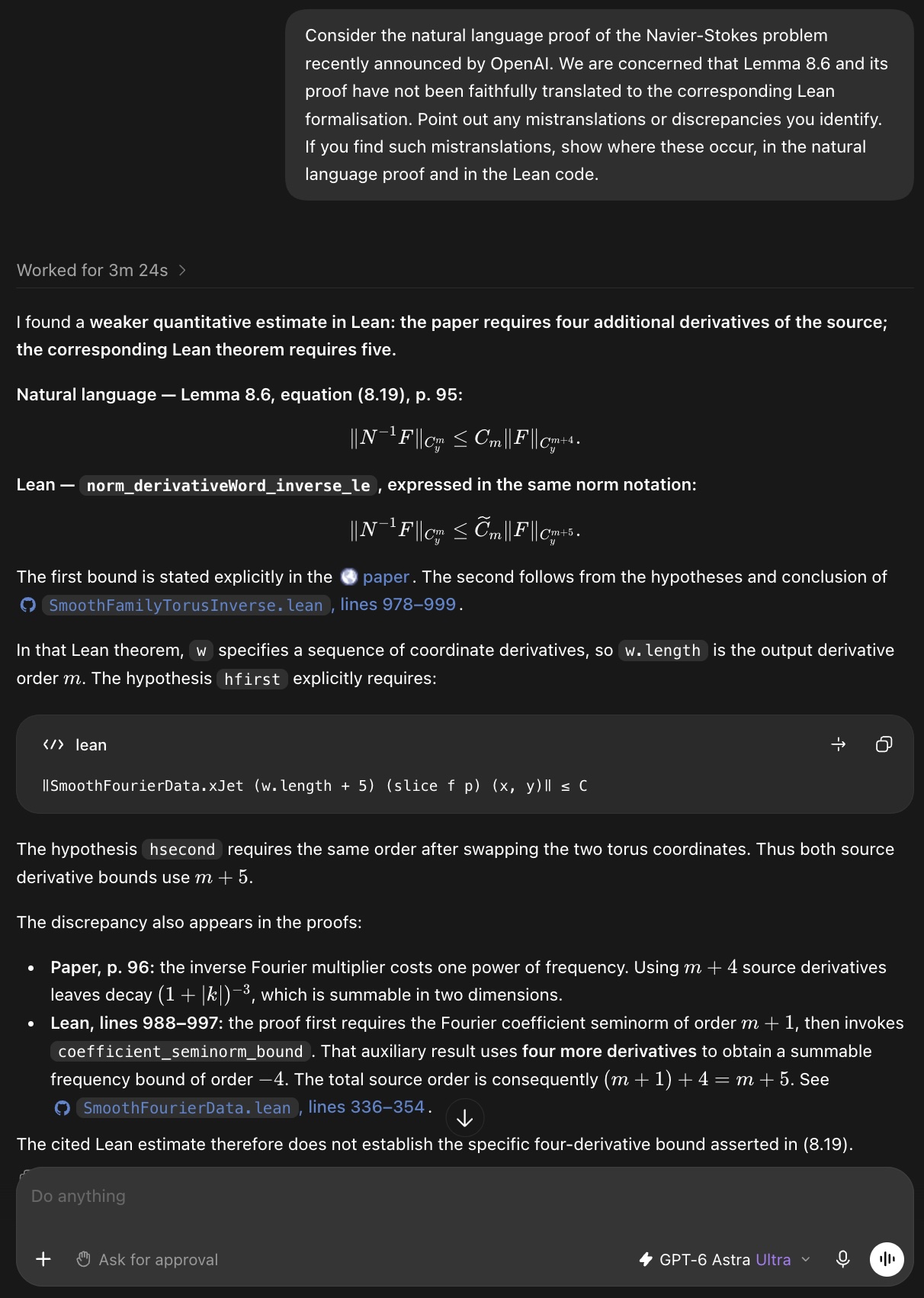}
\caption{ChatGPT-6 (Astra Ultra) shares our concern that the NL argument in Lemma 8.6 in the announced proof of blow-up of Navier--Stokes has not been faithfully translated into Lean. In particular, there is a discrepancy between Lemma~8.6 and the cited Lean estimate: the latter uses one additional input derivative to control the same torus derivatives of the output. See Example~\ref{ex:Navier_strong} for details.
}
\label{fig:Navier1}
\end{figure}

\subsection{When the NL paper declares stronger statements than what Lean proves} We commence with an example that complements those in \S \ref{sec:first}. In the following example the AI autoformalisation results in a different Lean proof of a weaker statement. 

\begin{example}[The NL paper claims a stronger inverse estimate than the cited Lean formalisation]\label{ex:Navier_strong}
Figure~\ref{fig:Navier1} concerns the estimate in
\href{https://cdn.openai.com/pdf/32d9f210-8b73-45e0-91bc-82a30aef8a9a/navier-stokes.pdf\#page=95}{Lemma~8.6, equation~(8.19)}, of OpenAI's \emph{Finite time blowup for Navier--Stokes}.
We compare this estimate with the cited Lean declarations at
\href{https://github.com/openai/NavierStokesAndEuler/tree/f9e8bc5b38b6e212696e8a30e3e91517af887bbd}{commit \texttt{f9e8bc5}} (click to see the documentation on Github).
The issue is a change in the order of the input derivatives required to bound the output derivatives. In particular, we have the following discrepancy. With $m$ a nonnegative integer, the comparison in Figure~\ref{fig:Navier1} is
\begin{align}
 \text{NL paper, equation (8.19):}\qquad
 &\|N^{-1}F\|_{C_y^m}\leq C_m\|F\|_{C_y^{m+4}}, \label{eq:Navier-inverse-comparison}
\\
 \text{Lean Estimate:}\qquad
 &\|N^{-1}F\|_{C_y^m}\leq K_m\|F\|_{C_y^{m+5}}. \label{eq:Navier-inverse-comparison-Lean}
\end{align}
\emph{Summary: The Lean code produced by OpenAI often has stronger conditions on bounds on the derivatives of various functions than the corresponding results in the NL paper}.

Above, $F : P \times \mathbb{T}^2 \to \mathbb{C}$ is a function of parameters in $P$ and a torus variable $y \in \mathbb{T}^2$, with $\mathbb{T}^2:= \mathbb{R}^2/\mathbb{Z}^2$. The non-negative constants $C_m$, $K_m$ are independent of $F$ but allowed to depend on $m$ and the fixed direction $v_t$ (which defines an operator $N = v_t \cdot \nabla_y$). Note the above norms only involve torus derivatives. See \S\ref{sec:NS1_notation} for further explanation of the notation.
\end{example}

The proof of Lemma~8.6 opens with the following passage (source paper,
\href{https://cdn.openai.com/pdf/32d9f210-8b73-45e0-91bc-82a30aef8a9a/navier-stokes.pdf\#page=96}{page~96}):
\begin{quote}
``The divisor bound (6.7) gives $|v_t\cdot k|^{-1}\leq C(1+|k|)$.
Four more derivatives of $F$ than the requested output leave a summable
$(1+|k|)^{-3}$ bound on the Fourier series in two dimensions.
This proves (8.19), smoothness, and uniqueness after fixing the zero Fourier
coefficient. The multiplier commutes with every slow coefficient derivative
and preserves reality.''
\end{quote}
The related Lean codes -- on the other hand -- all use five additional derivatives, not four. A key example is the following Lean theorem, which appears to bear the closest resemblance to (8.19):
\begin{lstlistinglean}
omit [FiniteDimensional ℝ P] in
theorem norm_derivativeWord_inverse_le (d : Direction) {f : Source P}
    (hf : ContDiff ℝ ∞ f) (hp : Periodic f) (w : List Bool) (p : P) (Y : Plane) {C : ℝ}
    (hzero : ∀ x ∈ Icc (0 : ℝ) 1, ∀ y ∈ Icc (0 : ℝ) 1, ‖f (p, (x, y))‖ ≤ C)
    (hfirst : ∀ x ∈ Icc (0 : ℝ) 1, ∀ y ∈ Icc (0 : ℝ) 1,
      ‖SmoothFourierData.xJet (w.length (*@\textcolor{red}{+ 5}@*)) (slice f p) (x, y)‖ ≤ C)
    (hsecond : ∀ x ∈ Icc (0 : ℝ) 1, ∀ y ∈ Icc (0 : ℝ) 1,
      ‖SmoothFourierData.xJet (w.length (*@\textcolor{red}{+ 5}@*))
        (SmoothFourierData.swapFunction (slice f p)) (x, y)‖ ≤ C) :
    ‖derivativeWord w (slice (inverse d f) p) Y‖ ≤
      ParametricTorusInverse.mixedLossConstant w.length * C := by
  have hbound := SmoothFourierData.coefficient_seminorm_bound (slice_smooth hf p) (hp p)
    (w.length + 1) hzero (by simpa only [Nat.add_assoc] using hfirst)
    (by simpa only [Nat.add_assoc] using hsecond)
  calc
    _ ≤ ((6 * ‖omega⁻¹‖) * ‖omega‖ ^ w.length) *
        coeffSeminorm (w.length + 1) (coefficient f p) :=
      inverse_derivativeWord_bound d
        (SmoothFourierData.rapid_coefficient (slice_smooth hf p) (hp p)) w Y
    _ ≤ ((6 * ‖omega⁻¹‖) * ‖omega‖ ^ w.length) *
        ((3 ^ ((w.length + 1) + 4) * C) * ∑' k : Frequency, (weight k ^ (*@\textcolor{red}{4}@*))⁻¹) :=
      mul_le_mul_of_nonneg_left hbound (by positivity)
    _ = _ := by simp only [ParametricTorusInverse.mixedLossConstant, Nat.add_assoc]; ring
\end{lstlistinglean}

\noindent{\footnotesize\linespread{1}\selectfont\raggedright\textbf{Source (GitHub permalinks):} \texttt{openai/NavierStokesAndEuler}, commit
\href{https://github.com/openai/NavierStokesAndEuler/tree/f9e8bc5b38b6e212696e8a30e3e91517af887bbd}{\texttt{f9e8bc5}}.\\
File: \texttt{NavierStokes/SmoothFamilyTorusInverse.lean}.\\
Declaration:
\href{https://github.com/openai/NavierStokesAndEuler/blob/f9e8bc5b38b6e212696e8a30e3e91517af887bbd/NavierStokes/SmoothFamilyTorusInverse.lean\#L1059-L1080}{\texttt{norm\_derivativeWord\_inverse\_le}, lines~1059--1080}.\\
Red has been added to highlight aspects relevant to this example. The displayed Lean code is otherwise unchanged. \par}

The most important aspects of the above Lean theorem \href{https://github.com/openai/NavierStokesAndEuler/blob/f9e8bc5b38b6e212696e8a30e3e91517af887bbd/NavierStokes/SmoothFamilyTorusInverse.lean#L1059-L1080}{\texttt{norm\_derivativeWord\_inverse\_le}} are the two hypotheses:  
\begin{lstlistinglean}
    (hfirst : ∀ x ∈ Icc (0 : ℝ) 1, ∀ y ∈ Icc (0 : ℝ) 1,
      ‖SmoothFourierData.xJet (w.length (*@\textcolor{red}{+ 5}@*)) (slice f p) (x, y)‖ ≤ C)
    (hsecond : ∀ x ∈ Icc (0 : ℝ) 1, ∀ y ∈ Icc (0 : ℝ) 1,
      ‖SmoothFourierData.xJet (w.length (*@\textcolor{red}{+ 5}@*))
        (SmoothFourierData.swapFunction (slice f p)) (x, y)‖ ≤ C) 
\end{lstlistinglean} 
\noindent which \footnote{The parameter set $P$ are the spatial variables which are not relevant to this part of the proof.} should be interpreted as
\begin{equation*}
	\left|\frac{\partial^{m+5} F(y_1,y_2)}{\partial y_{1}^{m+5}}\right| \leq C, \quad \left|\frac{\partial^{m+5} F(y_1,y_2)}{\partial y_{2}^{m+5}}\right| \leq C,
\end{equation*}
for every $y=(y_1,y_2) \in \mathbb{T}^2$, and the conclusion
\begin{lstlistinglean}
    ‖derivativeWord w (slice (inverse d f) p) Y‖ ≤
      ParametricTorusInverse.mixedLossConstant w.length * C
\end{lstlistinglean}
\noindent which\footnote{Strictly speaking, $w$ in OpenAI's Lean code is defined to be an object in the derivativeWord class: a list of booleans for which each `false' entry corresponds to a derivative in the $y_1$ coordinate and each `true' entry a derivative in the $y_2$ coordinate.} should be interpreted as 
\begin{equation*}
	\left|\frac{\partial^{|w|} N^{-1} F(y_1,y_2)}{\partial w}\right| \leq K_{|w|} \cdot C,
\end{equation*}
with $C$ the same constant as in the hypotheses, and where $w$ is a multi-index. Technically, to obtain an estimate directly comparable with equation (8.19) in OpenAI's NL paper, one would need also to take a supremum over all such all multi-indices $w$ with $|w| \leq m$ \footnote{This is not explicitly done as part of \texttt{norm\_derivativeWord\_inverse\_le}.}, thereby obtaining 
\begin{equation*}
	\left\|N^{-1} F\right\|_{C^m_y} \leq \hat K_m\left\|F\right\|_{C^{m+5}_y},
\end{equation*}
with $\hat K_m = \max(K_0,K_1,\dotsc,K_m)$. This estimate is then strictly weaker than equation (8.19) due to the extra derivative on the right-hand side.

\begin{remark}[Consistent mismatch between $m+4$ and $m+5$ derivatives]
Similar estimates as~\eqref{eq:Navier-inverse-comparison-Lean} can be seen throughout OpenAI's Lean code: for example, 
\[
\href{https://github.com/openai/NavierStokesAndEuler/blob/f9e8bc5b38b6e212696e8a30e3e91517af887bbd/NavierStokes/SmoothFamilyTorusInverse.lean#L865-L875}{\texttt{inverse\_finiteJets}} \, \text{ and } \, \href{https://github.com/openai/NavierStokesAndEuler/blob/f9e8bc5b38b6e212696e8a30e3e91517af887bbd/NavierStokes/ParametricTorusInverse.lean#L654-L680}{\texttt{mixedJet\_inverse\_bound}}
\] 
each contain a similar $m+5$ derivative requirement, rather than the expected $m+4$ derivatives from the NL paper, as in~\eqref{eq:Navier-inverse-comparison}.
\end{remark}
 
The reason for the above discrepancy is related to the difference in proof method employed by the NL paper and the Lean code. As the quote above indicates, the NL paper makes use of the convergent series $\sum_{k \in \Z^2} \frac{1}{\left(1+|k|\right)^3}$ by using $m+4$ derivatives. In contrast, the Lean code displayed above for 
\begin{equation}\label{eq:Lean1}
 \href{https://github.com/openai/NavierStokesAndEuler/blob/f9e8bc5b38b6e212696e8a30e3e91517af887bbd/NavierStokes/SmoothFamilyTorusInverse.lean#L1059-L1080}{\texttt{norm\_derivativeWord\_inverse\_le}}
\end{equation}
employs the series $\sum_{k \in \Z^2} \frac{1}{(1+|k_1|+|k_2|)^4}$ (see the exponent `4' highlighted in red in the code \eqref{eq:Lean1} above) and thus uses $m+5$ derivatives instead. Tracing the proof of \eqref{eq:Lean1} we find that the series arises from applying the Lean theorem \href{https://github.com/openai/NavierStokesAndEuler/blob/f9e8bc5b38b6e212696e8a30e3e91517af887bbd/NavierStokes/SmoothFourierData.lean#L368-L387}{\texttt{coefficient\_seminorm\_bound}}, 
just as Figure~\ref{fig:Navier1} mentions. Consequently, the Lean results discussed in this section are weaker \footnote{In fact, some such Lean results are actually incomparable with~\eqref{eq:Navier-inverse-comparison}, due to the mismatch in the type of derivative. For example, the Lean theorem \texttt{inverse\_finiteJets} uses mixed derivatives, not just torus derivatives. Note however that these results are still using $m+5$ derivatives.} than~\eqref{eq:Navier-inverse-comparison} in the NL proof.

\subsubsection{Explanation of the notation and technical information} \label{sec:NS1_notation}
The variable $y=(y_1,y_2)$ belongs to a two-dimensional torus
$\mathbb{T}^2=\mathbb{R}^2/\mathbb{Z}^2$. The fixed vector and differential operator in Lemma~8.6 are
\[
 v_t=(\sqrt{2}-1,1),\qquad
 N=v_t\cdot\partial_y=(\sqrt{2}-1)\partial_{y_1}+\partial_{y_2}.
\]
The subscript $t$ is specified in
\href{https://cdn.openai.com/pdf/32d9f210-8b73-45e0-91bc-82a30aef8a9a/navier-stokes.pdf\#page=63}{equation~(6.2)}
of the source paper. Here $F:\mathbb{T}^2\to\mathbb{C}$ is smooth, with zero normalized Haar mean: $\int_{[0,1]^2}F(y)\,\mathrm{d}y=0$. The notation $N^{-1}F$ denotes the unique smooth zero-mean
solution $\varphi$ of $N\varphi=F$. The zero-mean requirement assures this solution is uniquely specified (else any solution, via addition of a constant, yields another). For smooth $G : \mathbb{T}^2 \to \mathbb{C}$ and a nonnegative integer $r$, the following norm is defined:
\[
 \|G\|_{C_y^r}
 :=\max_{\substack{a,b\geq0\\a+b\leq r}}
       \sup_{y\in\mathbb{T}^2}
       |\partial_{y_1}^{a}\partial_{y_2}^{b}G(y)|,
 \qquad a,b \in \mathbb{N}.
\]
The subscript $y$ means that other variables are held fixed. This is relevant in the case $G$ is of the form $G(y) = F(p, y)$ for some $F : P \times \mathbb{T}^2 \to \mathbb{C}$ and some fixed $p$.

\subsection{When a statement and its proof in the NL paper are both mistranslated into Lean} We now give an example that goes further than Example \ref{ex:example2} in \S \ref{sec:first}. In the following example the AI autoformalisation results in a Lean proof of a different bound obtained by a substantively different mathematical argument.

\begin{figure}
\centering
\includegraphics[width=0.9\linewidth]{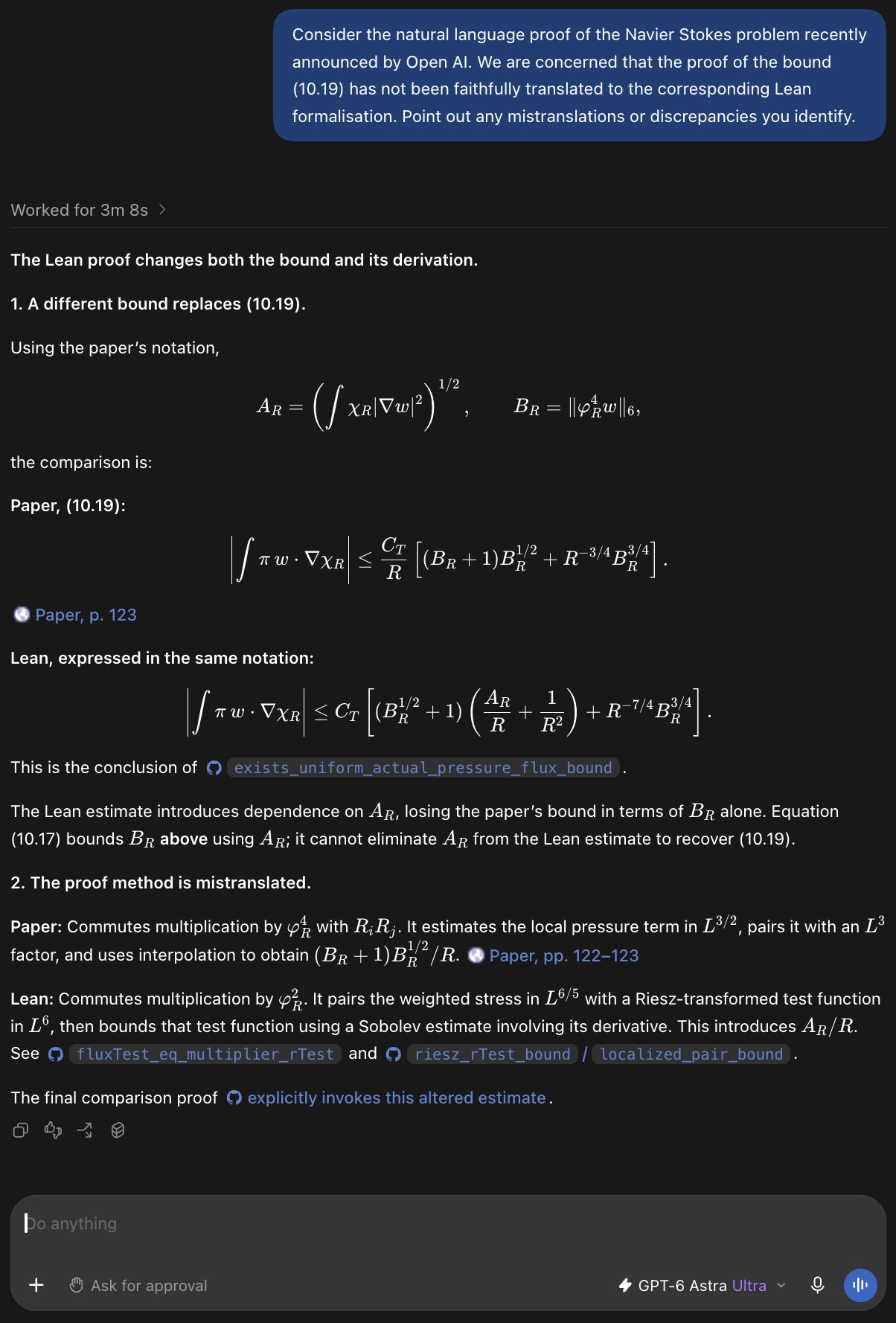}\caption{ChatGPT-6 (Astra Ultra) shares our concern that the proof of the pressure-flux bound~(10.19) in the announced proof of blow-up of Navier--Stokes is not semantically faithfully translated into Lean. Both the stated bound and the mathematical argument used to prove it differ from the cited Lean estimate. Example~\ref{ex:Navier_weak} explains this in further detail, including the notation.
}
\label{fig:Navier2}
\end{figure}

\begin{example}[The NL paper's pressure-flux bound and its proof are both mistranslated]\label{ex:Navier_weak}
Figure~\ref{fig:Navier2} concerns the estimate in
\href{https://cdn.openai.com/pdf/32d9f210-8b73-45e0-91bc-82a30aef8a9a/navier-stokes.pdf\#page=123}{equation~(10.19)}
of OpenAI's \emph{Finite time blowup for Navier--Stokes}.
We compare this estimate, together with its proof, with the cited Lean declaration at
\href{https://github.com/openai/NavierStokesAndEuler/tree/f9e8bc5b38b6e212696e8a30e3e91517af887bbd}{commit \texttt{f9e8bc5}}.
The issue here is not only a change in the bound, but also a change in the mathematical argument used to establish it. In particular, we have the following discrepancy: with $R\geq1$, the comparison in Figure~\ref{fig:Navier2} is
\begin{align} 
 \text{NL, equation (10.19):}\,
 &\Bigl|\int \pi\, w\cdot\nabla\chi_R\Bigr|
 \leq \frac{C_T}{R}\Bigl[(B_R+1)B_R^{1/2}+R^{-3/4}B_R^{3/4}\Bigr], \text{ for a.e. } t \in (0, T),
 \label{eq:Navier-flux-comparisonNL}\\[2mm]
 \text{Lean estimate:}\,
 &\Bigl|\int \pi\, w\cdot\nabla\chi_R\Bigr|
 \leq C_T\Bigl[(B_R^{1/2}+1)\Bigl(\frac{A_R}{R}+\frac{1}{R^2}\Bigr)+R^{-7/4}B_R^{3/4}\Bigr], \forall t \in (0,T) \label{eq:Navier-flux-comparisonLean},
\end{align}
where both the left hand side and the right hand side of both \eqref{eq:Navier-flux-comparisonNL} and \eqref{eq:Navier-flux-comparisonLean} depend on $t$; see below for explanation of the notation. 
The NL paper's bound is stated purely in terms of $B_R$, whereas the Lean estimate additionally depends on $A_R$. Here $C_T$ is a finite nonnegative constant, allowed to depend on $T$ (which itself is the maximum timescale for which the bounds are considered) and the norm conventions, but not on $R$. Although \href{https://cdn.openai.com/pdf/32d9f210-8b73-45e0-91bc-82a30aef8a9a/navier-stokes.pdf\#page=122}{equation~(10.17)} in the NL paper bounds $B_R$ above using $A_R$, it cannot be used to eliminate $A_R$ from the Lean estimate to recover~(10.19); the two bounds in \eqref{eq:Navier-flux-comparisonNL} and \eqref{eq:Navier-flux-comparisonLean} are therefore hard to compare. \footnote{In fact, the purpose of the lemma is to establish that $w = 0$ and hence the (equal to $0$) bound on the right hand side of \eqref{eq:Navier-flux-comparisonNL} is stricter than the (non-zero) bound on the right hand side of \eqref{eq:Navier-flux-comparisonLean}. However, this is only established after completing the entire proof of Lemma 10.5.}

The closest Lean code to (10.19) (equation \eqref{eq:Navier-flux-comparisonNL} in our paper) is the following result:
\begin{lstlistingleansmall}
theorem exists_uniform_actual_pressure_flux_bound {T : ℝ} {u v : VelocityField}
    {p q : PressureField} (H : PressureRecovery.Hypotheses T u v p q)
    (M(*@\ensuremath{{}_0}@*) U(*@\ensuremath{{}_0}@*) G(*@\ensuremath{{}_0}@*) : ℝ (hM(*@\ensuremath{{}_0}@*) : 0 ≤ M(*@\ensuremath{{}_0}@*)) (hU(*@\ensuremath{{}_0}@*) : 0 ≤ U(*@\ensuremath{{}_0}@*)) (hG(*@\ensuremath{{}_0}@*) : 0 ≤ G(*@\ensuremath{{}_0}@*))
    (hM : ∀ t ∈ Icc 0 T, MemLp (fun x => (u - v) (t, x)) 2 volume ∧
      comparisonLpNorm 2 (fun x => (u - v) (t, x)) ≤ M(*@\ensuremath{{}_0}@*))
    (hU : ∀ t ∈ Icc 0 T, MemLp (fun x => u (t, x)) 3 volume ∧
      comparisonLpNorm 3 (fun x => u (t, x)) ≤ U(*@\ensuremath{{}_0}@*))
    (hG : ∀ t ∈ Icc 0 T, ∀ i j : Fin 3,
      Integrable (tensorDiff u v t i j) volume ∧ comparisonLpNorm 1 (tensorDiff u v t i j) ≤ G(*@\ensuremath{{}_0}@*)) :
    ∃ CP : ℝ, 0 ≤ CP ∧ ∀ R : ℝ, 1 ≤ R → ∀ t ∈ Ioo 0 T,
      |(*@\ensuremath{\int}@*) x, (p - q) (t, x) * fderiv ℝ (ComparisonCutoffs.weight R) x ((u - v) (t, x))| ≤ CP *
        ((cutoffL6 (ComparisonCutoffs.cutoff R) (u - v) t ^ (1 / 2 : ℝ) + 1) *
          (dissipationRoot (ComparisonCutoffs.cutoff R) (u - v) t / R + 1 / R ^ 2) +
          R ^ (-(7 / 4 : ℝ)) * cutoffL6 (ComparisonCutoffs.cutoff R) (u - v) t ^ (3 / 4 : ℝ)) := by
obtain (*@\ensuremath{\langle}@*)CP, hCP, hbound(*@\ensuremath{\rangle}@*) := exists_uniform_canonicalCutoffFlux_bound M(*@\ensuremath{{}_0}@*) U(*@\ensuremath{{}_0}@*) G(*@\ensuremath{{}_0}@*) hM(*@\ensuremath{{}_0}@*) hU(*@\ensuremath{{}_0}@*) hG(*@\ensuremath{{}_0}@*)
refine (*@\ensuremath{\langle}@*)CP, hCP, ?_(*@\ensuremath{\rangle}@*)
intro R hR t ht
have ht' : t ∈ Icc 0 T := Ioo_subset_Icc_self ht
have hu := NavierStokes.PeriodicUniqueness.spatial_smooth H.smooth_u ht'
have hv := NavierStokes.PeriodicUniqueness.spatial_smooth H.smooth_v ht'
exact(actual_flux_integrable_and_le_canonicalNorm H ht (zero_lt_one.trans_le hR) hu hv).2.trans
  (hbound R hR u v t hu hv (hM t ht').1 (hU t ht').1 (fun i j => (hG t ht' i j).1)
    (hM t ht').2 (hU t ht').2 (fun i j => (hG t ht' i j).2))
\end{lstlistingleansmall}
\noindent{\footnotesize\linespread{1}\selectfont\raggedright
	\textbf{Source (GitHub permalinks):} \texttt{openai/NavierStokesAndEuler}, commit
	\href{https://github.com/openai/NavierStokesAndEuler/tree/f9e8bc5b38b6e212696e8a30e3e91517af887bbd}{\texttt{f9e8bc5}}.\\
	File: \texttt{NavierStokes/R3/PressureFlux.lean}.\\
	Declaration:
	\href{https://github.com/openai/NavierStokesAndEuler/blob/f9e8bc5b38b6e212696e8a30e3e91517af887bbd/NavierStokes/R3/PressureFlux.lean\#L576}{\texttt{exists\_uniform\_actual\_pressure\_flux\_bound}, line~576}.\\\par}

\par\smallskip
	
Note that similar bounds (which are nonetheless distinct from \eqref{eq:Navier-flux-comparisonLean}) appear in 
\[
\text{\href{https://github.com/openai/NavierStokesAndEuler/blob/f9e8bc5b38b6e212696e8a30e3e91517af887bbd/NavierStokes/R3/PressureFlux.lean\#L546-L572}{\texttt{exists\_uniform\_canonicalCutoffFlux\_bound}}} 
\]
and  \href{https://github.com/openai/NavierStokesAndEuler/blob/f9e8bc5b38b6e212696e8a30e3e91517af887bbd/NavierStokes/R3/PressureFlux.lean\#L516-L543}{\texttt{canonicalCutoffFlux\_bound}}\footnote{In addition, there is a similar argument that appears in the file \texttt{NavierStokes/R3StressPressureEstimate.lean} as part of the definition of \href{https://github.com/openai/NavierStokesAndEuler/blob/f9e8bc5b38b6e212696e8a30e3e91517af887bbd/NavierStokes/R3StressPressureEstimate.lean\#L38}{\texttt{regularized\_flux\_Bound}}. Once again the bound established there is different from the one in \eqref{eq:Navier-flux-comparisonNL}, and, intriguingly, does not appear to be used to prove the Navier Stokes theorem proposed by OpenAI. It is instead proven as one of the steps in the Lean compilation phase, but it is not directly used.}. The bound \eqref{eq:Navier-flux-comparisonLean} is given in the following Lean code, taken from that of 	\href{https://github.com/openai/NavierStokesAndEuler/blob/f9e8bc5b38b6e212696e8a30e3e91517af887bbd/NavierStokes/R3/PressureFlux.lean#L585-L589}{\texttt{exists\_uniform\_actual\_pressure\_flux\_bound}, Line~576}.

\begin{lstlistingleansmall}
      |(*@\ensuremath{\int}@*) x, (p - q) (t, x) * fderiv ℝ (ComparisonCutoffs.weight R) x ((u - v) (t, x))| ≤ CP *
        ((cutoffL6 (ComparisonCutoffs.cutoff R) (u - v) t ^ (1 / 2 : ℝ) + 1) *
          (dissipationRoot (ComparisonCutoffs.cutoff R) (u - v) t / R + 1 / R ^ 2) +
          R ^ (-(7 / 4 : ℝ)) * cutoffL6 (ComparisonCutoffs.cutoff R) (u - v) t ^ (3 / 4 : ℝ))
\end{lstlistingleansmall}
The discrepancy in \eqref{eq:Navier-flux-comparisonLean} above arises through the diverging proof methods employed by the NL paper and the Lean code. It is beyond the scope, and besides the point, of this article to go into the full technical details of the proof technique for the NL paper's estimate (10.19) (in this paper, \eqref{eq:Navier-flux-comparisonNL}). More information is provided below, but in summary three of the main reasons for the differences between the results and strategy used by the NL paper and the Lean argument are as follows:
\begin{enumerate}
	\item The Lean proof does not rely on the boundedness of the Riesz transform as a map from $L^{3/2} \to L^{3/2}$; it instead invokes Sobolev embedding to move to $L^2$ and uses the boundedness of the Riesz transform in $L^2$. The additional derivative involved in the Sobolev embedding is why \eqref{eq:Navier-flux-comparisonLean} features $A_R$.
	\item The Lean proof applies H\"older's inequality with different exponents, which changes the bounds at the end.
	\item In the Lean proof, the left hand integral in \eqref{eq:Navier-flux-comparisonNL} is explicitly shown to be continuous in time. This ensures that the statement for almost every $t \in (0,T)$ can be upgraded to a statement that avoids the `almost every' qualifier.
\end{enumerate}

\end{example}

\subsubsection{Explanation of the notation and technical information}

For context, inequality (10.19) appears itself in the proof of a velocity uniqueness result, namely Lemma 10.5. 
In the proof of Lemma 10.5, using the language of 
the NL paper, one chooses $0 \leq \varphi \leq 1 $ to be smooth,
compactly supported and equal to one on the unit ball. For $R \geq 1$, they write $\varphi_R(x) = \varphi(x/R)$ and
$\chi_R = \varphi^8_R$. Continuing, the NL paper states:  `\emph{Let $R_i$ be the Riesz transform with Fourier multiplier operator $i\xi_i/\|\xi\|_2$.}' 
The NL proof of Lemma 10.5 begins by assuming that $(v,P)$ and $(u,p)$ are two solutions -- with $v$ (respectively $u$) the velocity and $P$ (respectively $p$) the pressure -- to a constructed Navier Stokes problem. The NL proof defines $w := v - u$ and $\pi := P-p$ and proceeds to establish the vanishing of $w$\footnote{The Lean proof actually uses $w = u-v$ and $\pi = p -P$, with the sign difference carried through in further definitions.}; the estimate (10.19) is a stepping stone towards this result. Note that $v,u,w$ all depend on a time parameter $t$ as well as spatial parameters; the time parameter is suppressed throughout the remaining discussion. The proof also defines a distribution $\pi_* = \sum_{i, j} R_i R_j g_{ij}$ where $R_i, R_j$ are Riesz transforms, and where $g = (g_{ij})_{i,j \in \{1,2,3\}}$ with $g_{ij} = v_iv_j - u_iu_j = w_iw_j + w_i u_j + u_i w_j$. 
As indicated
in Figure~\ref{fig:Navier2}, using the NL paper's notation,
\[
A_R:=\Bigl(\int\chi_R|\nabla w|^2\Bigr)^{1/2},
\qquad
B_R:=\|\varphi_R^{4}w\|_{L^6},
\]
so that $A_R$ is a localized $L^2$ gradient norm and $B_R$ a localized
$L^6$ velocity norm.

The purpose of defining $\pi_*$ is that it is an attempt to get an explicit approximation for $\pi$ in terms of the velocity fields; hence instead of directly analysing the integrand in (10.19), one replaces $\pi$ in this term by $\pi_*$\footnote{A large part of the NL proof of Lemma 10.5 addresses this issue; ultimately, it suffices to show that $\nabla \pi = \nabla \pi_*$ in the sense of distributions.}. The integrand is then factored into a product of two terms: $\varphi_{R}^4 \pi_*$ and another test function containing $\varphi_{R}^{3}$. The Lean proof differs here, in that it splits the integral into $\varphi_{R}^2 \pi_*$ and another test function containing $\varphi_R^5$; we shall see the consequences of this change later.

The NL proof now tries to bound $\varphi_R^4 \pi_*$. The strategy here is to move $\varphi_R^4$ into the argument of the Riesz transformations in the definition of $\pi_*$; doing so ensures that these Riesz transforms are being applied to functions with compact support in $L^p$ spaces. In doing so one must pay a penalty, since $a R_i b \neq R_i a b$ in general for functions $a,b$. Hence (10.18) in the NL paper establishes that
\begin{equation} \label{eq:NLSplitting}
	\varphi_R^4 \pi_* = \sum_{i, j} R_i R_j (\varphi_R^4 g_{ij}) + \sum_{i,j} [\varphi_R^4,R_i R_j] g_{ij}\end{equation}
where $[\cdot,\cdot]$ is the standard commutator.

There are a number of differences in the approach the Lean argument uses here. As mentioned before, the proof in Lean takes out a factor of $\varphi_R^2$ instead of $\varphi_R^4$. Mathematically, the Lean argument is the following calculation \footnote{Lean takes the real parts of many terms in this bound, but this is not especially important for our exposition.}
\begin{align} 
	\int \pi w \cdot \nabla \chi_R &= \sum_{i, j \in \{1,2,3\}} \int g_{ij} R_i R_j (w \cdot \nabla \varphi_R^8) =   \sum_{i, j \in \{1,2,3\}} \int g_{ij} R_i R_j (w \cdot 8\varphi^2_R\varphi^5_R \nabla \varphi_R) \notag\\&=   8\sum_{i, j} \int \varphi_R^2 g_{ij} R_i R_j (w \cdot \varphi_R^5 \nabla \varphi_R)  + 8 \sum_{i, j} \int g_{ij} [R_i R_j,\varphi_R^2] (w \cdot \varphi_R^5 \nabla \varphi_R)  \label{eq:LeanSplitting}
\end{align}
where the first equality is the Lean code \href{https://github.com/openai/NavierStokesAndEuler/blob/main/NavierStokes/R3/PressureFlux.lean#L211-L219}{\texttt{actual\_flux\_eq\_canonicalCutoffFlux}}, the second equality is contained in the Lean code \href{https://github.com/openai/NavierStokesAndEuler/blob/f9e8bc5b38b6e212696e8a30e3e91517af887bbd/NavierStokes/R3/PressureFluxTest.lean#L26-L34}{\texttt{multiplier\_r\_eq\_weight\_deriv}}, as well as the code in \href{https://github.com/openai/NavierStokesAndEuler/blob/f9e8bc5b38b6e212696e8a30e3e91517af887bbd/NavierStokes/R3/PressureFluxTest.lean#L284-L291}{\texttt{cutoffTest} and \texttt{cutoffTest\_flux\_identity}} and the third equality uses the Lean codes given by the Lean theorem  \href{https://github.com/openai/NavierStokesAndEuler/blob/f9e8bc5b38b6e212696e8a30e3e91517af887bbd/NavierStokes/R3/PressureFlux.lean#L251-L267}{\texttt{pressurePair\_decomposition}} and  \href{https://github.com/openai/NavierStokesAndEuler/blob/f9e8bc5b38b6e212696e8a30e3e91517af887bbd/NavierStokes/R3/PressureFlux.lean#L372-L377}{\texttt{norm\_cutoff\_pressurePair\_le\_local\_commutator}}.

Both the NL proof and Lean proof then attempt to bound the first term in \eqref{eq:NLSplitting} and \eqref{eq:LeanSplitting} respectively\footnote{We do not comment here on how the second sum is treated by Lean and the NL, since differences in the bound and approach can easily be seen by observing the first sum.}. Here, there are substantial differences. The NL proof attempts to obtain an $L^{3/2}$ bound on the first term. It uses H\"older's inequality to show that $\varphi_R^4 g_{ij} \in L^{3/2}$ and the $L^{3/2} \to L^{3/2}$ boundedness of the Riesz transforms (which is established by citing \cite{SteinSingular}) to get a bound of the form $\|R_i R_j (\varphi^4_R g_{ij})\|_{3/2} \leq C \|\varphi^4_R g_{ij}\|_{3/2} \leq C' (\|\varphi^4_Rw\|_6+1)\|w\|_2 \leq C'' (B_R + 1)$, where $C,C',C''$ are all (potentially different) constants that can depend on $T$.

By contrast, the Lean proof does not rely on the above citation\cite{SteinSingular}. Instead, it begins by using  H\"older's inequality (with H\"older conjugates $6$ and $6/5$) to obtain 
\[\left|\int \varphi_R^2 g_{ij} R_i R_j (w \cdot \varphi_R^5 \nabla \varphi_R)\right|\leq \|\varphi_R^2 g_{ij}\|_{6/5} \|R_i R_j (w \cdot \varphi_R^5 \nabla \varphi_R)\|_6.\]

We will first look at how the Lean argument bounds the right hand term $\|R_i R_j (w \cdot \varphi_R^5 \nabla \varphi_R)\|_6$ (as in the lean code \href{https://github.com/openai/NavierStokesAndEuler/blob/f9e8bc5b38b6e212696e8a30e3e91517af887bbd/NavierStokes/R3/RieszTestOperators.lean#L223-L239}{\texttt{lpNorm\_six\_rieszTest\_le}} and \href{https://github.com/openai/NavierStokesAndEuler/blob/f9e8bc5b38b6e212696e8a30e3e91517af887bbd/NavierStokes/R3/PressureFlux.lean#L289-L305}{\texttt{riesz\_rTest\_bound}}). First, a Sobolev embedding theorem (bounding the $L^6$ norm by a multiple of the $L^2$ norm of the derivative, because this is an argument done over the spatial variables which are in three dimensions) is used, as in the file \href{https://github.com/openai/NavierStokesAndEuler/blob/f9e8bc5b38b6e212696e8a30e3e91517af887bbd/NavierStokes/R3/SmoothSobolevL6.lean#L140-L150}{\texttt{smooth\_eLpNorm\_six\_toReal\_le}}. Next, an interchange of derivatives with the Riesz transform (which comes from the Fourier multipliers) as stated in both \href{https://github.com/openai/NavierStokesAndEuler/blob/f9e8bc5b38b6e212696e8a30e3e91517af887bbd/NavierStokes/R3/RieszTestOperators.lean#L89-L100}{\texttt{pderiv\_rieszTest}} and \href{https://github.com/openai/NavierStokesAndEuler/blob/f9e8bc5b38b6e212696e8a30e3e91517af887bbd/NavierStokes/R3/RieszTestOperators.lean#L102-L104}{\texttt{partial\_rieszTest}} is employed to move the gradient operator inside the transform. Finally, the argument uses the boundedness of the Riesz operators from Schwartz functions in $L^2$ to functions in $L^2$ (which follows easily from considering the Fourier multipliers of these operators) as in each of \href{https://github.com/openai/NavierStokesAndEuler/blob/f9e8bc5b38b6e212696e8a30e3e91517af887bbd/NavierStokes/R3/RieszL2Bounds.lean}{\texttt{integral\_norm\_sq\_rieszTest\_le}}, \href{https://github.com/openai/NavierStokesAndEuler/blob/f9e8bc5b38b6e212696e8a30e3e91517af887bbd/NavierStokes/R3/RieszTestOperators.lean#L184-L214}{\texttt{lpNorm\_two\_fderiv\_rieszTest\_le}} and \href{https://github.com/openai/NavierStokesAndEuler/blob/f9e8bc5b38b6e212696e8a30e3e91517af887bbd/NavierStokes/R3/RieszL2Bounds.lean#L165-L169}{\texttt{rieszTest\_memLp\_and\_l2\_bound}}. Symbolically, the calculation is as follows:
\begin{align*} \|R_i R_j (w \cdot \varphi_R^5 \nabla \varphi_R)\|_6 \leq C  \|\nabla R_{i}R_j (w \cdot \varphi_R^5 \nabla \varphi_R)\|_2 \leq C' \bigg\| R_{i}R_j \nabla (w \cdot \varphi_R^5 \nabla \varphi_R)\bigg\|_2\leq C'\| \nabla(w \cdot \varphi_R^5 \nabla \varphi_R)\|_2,
\end{align*}
where $C$ is a constant from the Sobolev inequality and $C'$ is another constant. In \href{https://github.com/openai/NavierStokesAndEuler/blob/f9e8bc5b38b6e212696e8a30e3e91517af887bbd/NavierStokes/R3/PressureFluxTest.lean#L139-L184}{\texttt{norm\_fderiv\_r\_le\_raw}}, the Lean proof then bounds (using the product rule), for $x \in \R^3$ (and fixing the time parameter) 
\begin{align*}
	\|\nabla(w \cdot \varphi_R^5 \nabla \varphi_R)(x)\|_{\ell_2} \leq & \,\varphi^5_R(x) \|\nabla \varphi_R(x)\|_{\ell_2} \|\nabla w(x)\|_{op} +  \varphi_R^5 \|\nabla^2 \varphi_R(x)\|_{op}\|w(x)\|_{\ell_2} \\&+ 5 \varphi_R(x)^4 \|\nabla \varphi_R(x)\|^2_{\ell_2} \|w(x)\|_{\ell_2}.
\end{align*}

Some further simplification is done (see for example \href{https://github.com/openai/NavierStokesAndEuler/blob/f9e8bc5b38b6e212696e8a30e3e91517af887bbd/NavierStokes/R3/PressureFluxTest.lean#L187-L225}{\texttt{norm\_fderiv\_r\_le\_amplitude}}), but ultimately this expression leads to the $A_R/R +1/R^2$ term present in \eqref{eq:Navier-flux-comparisonLean}.
When $\varphi_R(x)$ was defined, it was done so that $\varphi_R(x) \in [0,1]$,  $\|\nabla \varphi_R(x)\|_{\ell_2} \leq C/R$ and $\|\nabla^2 \varphi_R(x)\|_{op} \leq C_2/R^2$. The $A_R$ term comes from moving $\varphi_R^4$ inside the norm in $\|\nabla w(x)\|_{op}$. 

Finally, we note that the $B_R^{1/2} + 1$ term that appears in the Lean statement \eqref{eq:Navier-flux-comparisonLean} comes from using the definition of $g_{ij}=  w_iw_j + w_i u_j + u_i w_j$ to bound $\|\varphi_R^2 g_{ij}\|_{6/5}$. Indeed, $|\varphi_{R}^2 g_{ij}|\leq 2|u||w| +|\varphi_R w|^2$ is established pointwise\footnote{Technically, the Lean theorem \texttt{norm\_weighted\_tensorDiff\_le} establishes the result $|s^2 g_{ij}| \leq 2|u||w| + |sw|^2$ for any $0 \leq s \leq 1$, which can then be applied with $s = \varphi_{R}(x)$ for each $x$.} in \href{https://github.com/openai/NavierStokesAndEuler/blob/f9e8bc5b38b6e212696e8a30e3e91517af887bbd/NavierStokes/R3/LocalizedTensorBounds.lean#L119} {\texttt{norm\_weighted\_tensorDiff\_le}}. The Lean proof bounds $\lVert |\varphi_{R}w|^{2} \rVert_{6/5} \leq \lVert \varphi_{R} w \rVert_{12/5}^{2}$, and then an interpolation inequality \href{https://github.com/openai/NavierStokesAndEuler/blob/f9e8bc5b38b6e212696e8a30e3e91517af887bbd/NavierStokes/R3/WeightedInterpolation.lean#L150-L161}{\texttt{cutoff\_interpolation\_twelve\_fifths}} is used to get \begin{equation} \label{eq:Lean125Interpolation}\|\varphi_R w\|_{12/5} \leq \|w\|_2^{3/4}\|\varphi^4_R w\|_6^{1/4} \leq C B_R^{1/4}.\end{equation} Secondly, to bound $\| |u||w| \|_{6/5}$, a simple application of H\"older in \href{https://github.com/openai/NavierStokesAndEuler/blob/f9e8bc5b38b6e212696e8a30e3e91517af887bbd/NavierStokes/R3/LocalizedTensorBounds.lean#L46}{\texttt{cross\_norm\_bound}} gives $\||u||w| \|_{6/5} \leq \|u\|_3 \|w\|_2 \leq C$. This argument is also distinct from the NL proof, as no $12/5$ interpolation like the one present in \eqref{eq:Lean125Interpolation} is contained in the proof of Lemma 10.5.

\begin{remark}[{\bf Further potential mistranslations of the Navier-Stokes proof}]\label{rem:further}
The above examples require careful manual checking of both the NL proof as well as the Lean proof, which is delicate and highly time consuming. Moreover, the fact that we display only two examples does not mean that these are the only cases of mistranslations. Thus, a pertinent question becomes: are there more examples of mistranslations, and how can one find them in the most efficient way? For a preliminary answer to this question, the reader may consult the following  \href{http://www.damtp.cam.ac.uk/research/afha/anders/Navier-Stokes_Experiment.pdf}{\underline{experiment} (click to view)} aimed at finding other mistranslations in OpenAI's Navier-Stokes proof. 
\end{remark}

\subsection{Mistranslations and OpenAI's announced Euler proof} Alongside its Navier-Stokes announcement, OpenAI reported finite-time blow-up for the unforced incompressible Euler equations on \(\mathbb{R}^{3}\), starting from smooth, compactly supported initial data. The Euler blow-up and well posedness question has a long history, and we will -- as in the Navier-Stokes case -- only mention a small subset of results here  \cite{BealeKatoMajda1984,ConstantinFeffermanMajda1996,LuoHou2014,Elgindi2021, ChenHou2022Analysis, ChenHou2025Numerics, AlpogeBuckmaster2026Euler, DeLellisSzekelyhidi2009DifferentialInclusion, CordobaMartinezZoroa2024IPM}.

The format of the announced Euler paper is similar to the Navier-Stokes. There is a NL proof and an autoformalised Lean proof. As in the Navier-Stokes case, it is natural to ask whether this formalisation in Lean represents a semantically faithful translation of the NL paper, so that we can conclude that the NL paper represents a correct proof.  It is beyond the scope of this paper to carefully examine the autoformalisation of the Euler proof in a similar way to what we have done in this paper with the Navier-Stokes case. However, based on some examinations of the autoformalisation of the Euler case -- with both AI and manual tools -- we think that the Euler proof will be another great case study for mistranslations in practice.    

We encourage other readers to examine the autoformalisation of OpenAI's Euler NL proof and determine if there are mistranslations. Our prediction is based on preliminary examinations, but our assessment is that the likelihood of substantial mistranslations is very high.

\subsection{Why would mistranslations happen?}

When success is judged solely by obtaining a Lean proof which successfully compiles, the procedure followed by the AI can be summarised as follows.

\begin{quote}
\textbf{Input:} A faithful Lean translation $X'$ of an NL theorem $X$ and an alleged NL proof $Y$ of $X$.
\begin{tabbing}
\textbf{while} no proof of $X'$ has been verified \textbf{do}\\
\quad\= The AI revises $Y'$ using $Y$ as a guide (that is, the AI attempts a `translation' of $Y$ into $Y'$)\\
\> The AI checks if $Y'$ is a valid Lean proof of $X'$.\\
\textbf{end while}\\
\textbf{Output}: the verified proof $Y'$.
\end{tabbing}
\end{quote}
The acceptance criterion tests whether $Y'$ proves $X'$ in Lean. It does not
check whether the reasoning in $Y$ has been preserved. A revision may
repair an error or replace a correct argument by another one, just as in the above examples. Both count as success under this criterion, and neither establishes fidelity to the original purported proof. In fact, since the above procedure accepts any valid Lean proof, if the original NL proof is indeed incorrect, the AI is incentivised to find a different proof.

\begin{remark}[Why AI cannot just translate the correct Lean proof back to English/LaTeX]
It is natural to ask whether these mistranslation issues could be avoided by simply back-translating the produced Lean proof into NL and comparing with the original NL proof. It then remains, however, to establish both that the back-translation semantically faithfully represents the formal proof, and that its reasoning corresponds to the original argument. That is, back-translation requires further checks to ensure correspondence; it does not escape the objections raised in this section. In other words, we reach the same issue of semantic faithfulness not present in the outline for \ref{item:TranslationNoSemantic} but available only with autoformalisation in \ref{item:TranslationSemantic} above, just going in the reverse direction. In summary, this approach is as difficult as type \ref{item:TranslationSemantic}.
\end{remark}

\section{Logic in AI: Why semantic preservation in AI autoformalisation is hard} \label{sec:SemanticPreservationHard}
For generations, logicians, mathematicians, linguists and philosophers have vigorously debated questions concerning natural and formal languages, semantics, and semantic preservation in translation  \cite{Moschovakis2006, Hales2019CNL, IcardMoss2023, BertoHornischer2023, vanBenthem2014, Montague1970UniversalGrammar, Montague1974English, Quine1959MeaningTranslation, Quine1970Indeterminacy, Tarski1944Truth, Rabe2017Logics, Benthem2}. For standard background on formal languages, interpretations and deductive systems, see \cite[Chapters~1--2]{Enderton2001}. In autoformalisation semantic preservation becomes particularly delicate because of the need to resolve ambiguities, as we will see below. Throughout, we use $\mathbb{N}=\{0,1,2,\ldots\}$, as in Lean's \texttt{Nat}.

\subsection{Logic in autoformalisation: Resolving ambiguities is harder than the Halting problem}\label{sec:Halting} A crucial problem facing any AI autoformaliser is the task of resolving ambiguities. Note that ambiguities may not always come from the fact that NLs produce ambiguous statements. Indeed, as the following example demonstrates, ambiguities can come from the fact that well-definedness of mathematical objects can depend on computational problems. This fact indicates that the hardness of resolving such an ambiguity is induced by the hardness of solving certain computational problems.  Hence, results from foundations of mathematics \& computation -- such as the SCI hierarchy/arithmetical hierarchy -- become important in order to understand the difficulty of resolving ambiguities. 
 
\begin{example}[A hidden existence presupposition causing ambiguities]\label{ex:example3}
	An author of a textbook constructs explicitly a recursive enumeration (that is, there exists a computer program that can list them) $\{p_e\}_{e \in \mathbb{N}}$ of polynomials in $k+1$ variables with integer coefficients. Later on, they write down a particular value of $e$, and define $n_e$ to be the least natural number $n$ for which
	$p_e(n,x_1,\ldots,x_k)=0$ has no solution in natural numbers $x_1,\ldots,x_k$. They then define the real number $r_e=1/(n_e+1)$ and make the claim that it is obvious that $(r_e+1)^2 = r_e^2+2r_e+1$.
\end{example}

At first glance, this may seem like a trivial example to autoformalise. Indeed, the conclusion that $(r_e+1)^2 = r_e^2 + 2r_e+1$ is elementary to establish in Lean. The problem is the following: For Lean to prove that $(r_e+1)^2 = r_e^2 + 2r_e+1$, then $r_e$ must be well-defined as a rational number. In particular, it must be established whether or not there exists an $n$ such that $p_e(n,x_1,\ldots,x_k)=0$ has no natural-number solution $x_1, \ldots, x_k$. This issue causes an ambiguity: can Example \ref{ex:example3} be faithfully translated to Lean at all? The answer is as follows: yes, if $r_e$ is well defined, and no if $r_e$ is not well defined (when $n_e$ does not exist).  An AI autoformaliser aiming to translate Example \ref{ex:example3} preserving the semantics is faced with the problem of resolving this ambiguity. If $r_e$ is not well-defined, the AI autoformaliser should not attempt to translate in order to make Lean accept the translation. 

\begin{remark}[A semantic preserving AI autoformaliser must refrain from translating in certain cases]\label{rem:remark_refrain}
 Note that if the AI autoformaliser attempts to `translate' Example \ref{ex:example3} when $r_e$ is not well defined -- and Lean accepts the `translation' -- the `translation' cannot be semantically faithful, it must be a mistranslation. 
 This begs the following question: 
\begin{quote}
	{\bf \emph{
		Can an AI autoformaliser resolve the ambiguity in Example \ref{ex:example3} in general? That is, given $e \in \mathbb{N}$, can an AI autoformaliser determine when $r_e$ is well defined? 
	}}
\end{quote}
The answer to this problem comes from Hilbert's 10th problem \cite{Matiyasevich1993} when $k \geq 9$:
\begin{quote}
	{\bf \emph{
		No, such an AI autoformaliser can never exist, even if one could solve the Halting problem. 
	}}
\end{quote}
See Appendix \ref{appendix1} and Theorem \ref{thm:presupposition} for a detailed explanation of this statement. Example \ref{ex:example3} shows why semantic preservation in mathematics is hard. In general, an AI cannot determine when a statement is ambiguous -- and thus cannot be translated -- and when a statement is not ambiguous and can be faithfully translated. Moreover, the above result can be strengthened and informally stated as follows:
 \begin{quote}
	\emph{
		Providing semantically faithful AI autoformalisation is harder than any computational problem. 
	}
\end{quote}
See Appendix \ref{sec:ambSCI} for a precise statement.
\end{remark}

\begin{example}[A natural attempt to resolve Example \ref{ex:example3}]\label{ex:example4} It is natural to ask why one cannot simply apply the following algorithm, which may fail to halt and hence does not directly contradict the negative result concerning Example \ref{ex:example3} above:
	\begin{quote}
		\textbf{Input:} Example \ref{ex:example3} and $e \in \mathbb{N}$ (recall that $e$ uniquely defines the claim in Example \ref{ex:example3}).
		\begin{tabbing}
			\textbf{while} no translation of Example~\ref{ex:example3} has been verified by Lean \textbf{do}\\
			\quad\= The AI revises an attempted translation $Y'$ of Example~\ref{ex:example3} into Lean.\\
			\> The AI checks if $Y'$ is accepted by Lean.\\
			\textbf{end while}\\
			\textbf{Output}: the accepted translation $Y'$.
		\end{tabbing}
	\end{quote}
At face value, it seems like this methodology would only accept a translation when $n_e$ is defined, and thus all choices of $e$ for which the claim that $(r_e+1)^2 = r_e^2 + 2r_e+1$ in Example \ref{ex:example3} is a true unambiguous statement would ultimately be resolved.
\end{example}

\subsection{Why the algorithm in Example \ref{ex:example4} cannot solve the problem}\label{sec:why}	
	
	There is once again a problem with this approach. The AI could easily produce an attempted translation $Y'$ which uses a Lean definition that supplies a default value, when $n_e$ does not exist, in its proof that $(r_e+1)^2 = r_e^2 + 2r_e+1$. In particular, for a fixed $e \in \mathbb{N}$, write for $n \in \mathbb{N}$ the predicate $B_e(n)$ as follows:
	\begin{equation}\label{eq:pred}
	B_e(n)\;\Longleftrightarrow\;
	\forall \,(x_1, \ldots, x_k) \in\mathbb N^k,\quad p_e(n,x_1, \ldots, x_k)\ne0.
	\end{equation}
	With this predicate, a `translation' of Example \ref{ex:example3} in the case where $k=2$ and 
	\begin{equation}\label{eq:pol}
	p_e(n,x_1,x_2) = x_1 + x_2 - n,
	\end{equation}
	 to Lean code (that compiles) can be done in the following way, with similar examples possible for a different choice of $e$ and thus $p_e$. In particular, suppose an AI suggested the following `translation' $Y'$ of Example~\ref{ex:example3} into Lean:
	
\begin{lstlistinglean}
import Mathlib

def p_e (n x1 x2 : Int) : Int := x1 + x2 - n

def B_e (n : Nat) : Prop :=
forall (x1 x2 : Nat), p_e (n : Int) (x1 : Int) (x2 : Int) != 0

noncomputable def n_e : Nat := sInf {n : Nat | B_e n}

noncomputable def r_e : Rat := 1 / (n_e + 1)

theorem example_3_1 : (r_e + 1) ^ 2 = r_e ^ 2 + 2 * r_e + 1 := by ring
\end{lstlistinglean}
	
	\vspace{2mm}
	
	Here \texttt{sInf} returns the least element of a nonempty set
	of natural numbers, but returns $0$ when the set is empty; see \texttt{Nat.sInf\_def} and \texttt{Nat.sInf\_empty} in \cite{MathlibNatInf}.
	Thus, if no natural number satisfies $B_e$, the code silently
	assigns $n_e=0$ and $r_e=1$. 
	However, for the polynomial in \eqref{eq:pol}, every $n\in\mathbb{N}$ admits the natural-number root $(x_1,x_2)=(n,0)$, since $p_e(n,n,0)=n+0-n=0$. Hence, $B_e(n)$ is false for every $n$, and 
	\[
	\{n\in\mathbb{N}:B_e(n)\}=\varnothing.
	\] In particular, $r_e$ is not well-defined. 
	The `proof' succeeds without resolving the ambiguity and the existence presupposition because it mistranslates the statement by using \texttt{sInf};

	{\bf Warning:} With the above `translation' $Y'$ (in the grey box) the algorithm in Example \ref{ex:example4} will halt and produce a mistranslation. In this case the semantics of the definition of $n_e$ have been completely changed. Indeed, if the above `translation' into Lean had actually been semantically faithful, it would have implied that there is an $n \in \mathbb{N}$ such that $p_e(n,x_1,x_2) = x_1 + x_2 - n$ has no roots $(x_1,x_2) \in \mathbb{N}^2$ -- this is false! 

\subsection{The root of the difficulty -- Creating semantic preserving autoformalisers is a major open problem}
All the problems discussed in the different examples above can be summarised in one rule of thumb, which illustrates the root of the problem:
	\begin{quote}
		{\bf  
			\emph{Lean acceptance (compilation) of translation does not imply semantic preservation.}
		}
	\end{quote}
To illustrate the issue we consider the following example:

\begin{example}[Lean compilation does not imply semantic preservation]\label{ex:23}
Consider the following two sentences: 
\begin{itemize}
\item[(1)] \emph{There are infinitely many odd numbers among the natural numbers.}
\item[(2)] \emph{There are infinitely many prime numbers among the natural numbers.}
\end{itemize}
Changing the word `odd' to `prime' changes the semantics of the first sentence completely, as it becomes the second sentence. Suppose an AI should translate sentence (1) into Lean. Suppose also that it did so semantically faithfully and Lean compiles. What would have happened if the AI instead had made a mistake and changed `odd' to `prime', and then translated sentence (2) semantically faithfully to Lean -- however, the intention was to translate sentence (1)? Lean would still compile, yet the `translation' is actually a mistranslation. 
\end{example}	

Example \ref{ex:23} reiterates the message in \S \ref{sec:why}: There are many candidate `translations' that would make Lean compile -- some of them would be semantically faithful translations -- but many would be mistranslations. Thus, it will be up to the AI to choose a semantically faithful translation, Lean does not provide a security net if the AI provides a mistranslation. This means that the AI is subject to the limitations arising from the foundations of mathematics and computations as highlighted in \S \ref{sec:Halting} and Theorem \ref{thm:presupposition}. Indeed, the AI autoformaliser cannot resolve all ambiguities and must therefore admit that there are many cases for which it will not know whether a translation is possible or not -- even if a semantic translation is possible.  

How to design trustworthy AI autoformalisers -- meaning that they will only provide semantically faithful translations -- is beyond the scope of this paper. Such developments will be announced in forthcoming articles. The key observation, however, is that even the most advanced models of OpenAI are far from achieving this goal. In particular, creating trustworthy AI autoformalisers that are capable of preserving semantics is a major open problem.


\section{Summary and perspectives}
\subsection{Attempting semantics-preserving autoformalisation in practice -- The Meta attempt}\label{sec:meta}
Here is a pertinent example of the second type of autoformalisation -- faithful semantic translation of a mathematical text -- going awry in practice.  During the spring of 2026, Meta declared a new milestone in autoformalisation: they claimed to have translated 26 mathematics textbooks into Lean, forming a library with 45,000 verified Lean declarations and around 500,000 lines of Lean code~\cite{RammalEtAl2026}. These claims were, however, swiftly refuted by the Lean community in the discussion quoted below~\cite{LeanCommunityAtlas2026}.
\begin{itemize}
	\item[(i)] Jack McCarthy: `\textit{Pg. 17 of the paper claims that 56\% of the target statements in this book [i.e. one of the 26 books] were formalized, but I could not find a single one which did not contain a fatal error like above. The reality is that 0\% of the statements here are accurately formalized}.'
	\item[(ii)] Brian Nugent: `\textit{I took a look at the algebraic geometry folder and it appears that the things it did prove (about 60\%
		of the book according to your website) were exactly the things that are in Mathlib and it failed on everything else. [...] claiming ``my AI autoformalized this AG textbook'' in the state that’s its in is simply lying}.'
	\item[(iii)] Christian Merten: `\textit{I would like to push back a bit on the point that this project has produced something useful on this front}.'
\end{itemize}
Fortunately, in this case, Lean experts were quickly able to detect the mistakes made by the AI in the autoformalisation. This may not be the case in the future:  proofs produced by autoformalisation are becoming longer and more complex. For example, Anthropic's recent autoformalisation of Fermat's Last Theorem took 13 million lines of Lean code~\cite{Buzzard2026}. The increases in both the number and length of autoformalised proofs will make it extremely difficult to catch mistranslations.

The problem of mistranslations by autoformalisation is particularly relevant in the context of building a unified body of knowledge in Lean, where results are consistent with, and build upon, one another. A single falsely translated statement in Mathlib means that all further results building upon this statement could be rendered invalid. Ensuring autoformalisation is done correctly is therefore of utmost importance to ensuring Lean libraries such as Mathlib remain correct and usable. 

We may ask: \emph{why did this autoformalisation project by Meta turn out the way it did?} The answer is that autoformalising AI systems currently follow something akin to the algorithm in Example \ref{ex:example4} to produce translations. The project did go beyond this, checking for faithfulness and non-cheating in the translations;  this semantic verification was, however, mainly performed by other AIs~\cite[Section~3 and Appendix~D]{RammalEtAl2026}. The fact that this failed highlights the difficulty that AIs have both with producing semantically faithful translations into Lean and with assessing the accuracy of such translations.

\begin{remark}
The crucial issue is that -- as pointed out in Remark \ref{rem:remark_refrain} -- a semantic preserving AI formaliser must refrain from doing a translation in certain cases. The example in \S \ref{sec:SemanticPreservationHard} and Theorem \ref{thm:presupposition} show that no AI autoformaliser can determine when it must refrain from translating and when it should translate. Thus, the big question is how to design AI autoformalisers that preserve semantics. Current AI based approaches are typically based on the `guessing and checking'  algorithm in Example \ref{ex:example4}, which is doomed to mistranslate because no AI can determine when it should or should not provide a suggested translation $Y'$ (recall the notation from Example \ref{ex:example4}).  
\end{remark}


\subsection{Semantics-preserving autoformalisation and the future of AI in mathematics}
The Leiden Declaration from June 2026~\cite{Leiden2026} discusses how incorrect use of AI can be harmful to mathematics as a discipline. It singles out in particular the dangers posed by autoformalisation:
\begin{quote}
	`\textit{Current automated techniques can produce plausible but unreliable (or even incorrect) arguments which are difficult to distinguish from correct mathematical proofs. \textbf{This applies not only to informal arguments, but also to formalizations, where the difficulty lies in the translation between computer-encoded and human presentations of concepts}}.'
\end{quote}
The boldface text refers exactly to the issues with the two forms of autoformalisation as discussed in this note. 
Even more recently, in September 2026, a declaration titled `A Severe Misalignment of AI in Mathematics'~\cite{Misalignment2026} initiated by 25 Fields medalists brings up other issues related to AI in mathematics:

\begin{quote}
	`\textit{But solving problems is only a tool and proxy for achieving the primary goal of conceptual understanding and insight. Forgetting this in the world of AI may turn the tool against the primary goal. Indeed, the mass production at faster and faster pace of ``true/false'' statements could destroy fertile ground instead of breathing life into new ideas.}'
\end{quote}
The quote also relates fundamentally to autoformalisation:
\begin{itemize}
\item[(i)] 
Solving a problem only formally in Lean -- without a semantically faithful translation from a correct NL argument that humans can read -- would typically only be a proxy for the optimal final product which should provide understanding of the argument (both an NL argument and a Lean proof  with a semantically faithful translation between the two). Without the semantically faithful translation, the NL argument could even be wrong -- which is detrimental to human understanding -- but would also harm future instalments of AI which rely on massive collections of correct mathematical NL literature in the training process.     

\item[(ii)]  Pure `true/false' statements coming from only Lean proofs (lacking a semantically faithful translation from a correct NL argument) can certainly be useful. Indeed, we would like to know if the Riemann Hypothesis is true. However, a principal reason for the interest in this problem stems from the belief that the techniques and ideas used in a proof could be applied broadly to address other questions and to enhance our understanding across analysis, number theory and the many other fields that the Riemann Hypothesis touches on. A correct but incomprehensible Lean proof would not have the same value -- mathematicians will unlikely read millions of lines of Lean code to try to understand. It would be even worse if an incorrect NL proof of the Riemann Hypothesis is accepted based on a mistranslated Lean proof; we would then have a complete misunderstanding of the facets of this problem.
\end{itemize}

As argued above, AI autoformalisers that can guarantee semantically faithful translations do not currently exist. Thus, developing semantically faithful AI autoformalisation stands as a major open problem, and attempts at quick fixes will typically violate fundamental limitations that date back to G\"odel and Turing.

\appendix

\section{The difficulty of resolving ambiguities and Hilbert's 10th problem}\label{appendix1}

\subsection{Help from the SCI hierarchy and Hilbert's 10th problem}

Here we continue the presupposition resolution discussion from Example~\ref{ex:example3}, explaining why no AI can resolve it for all $e \in \mathbb{N}$. The basic framework we use is that of the Solvability Complexity Index (SCI) hierarchy initiated in \cite{Hansen_JAMS} and further developed in \cite{CRAS, SCI, colbrook2019foundations, GazdagBastounisHansen2026, ColbrookAntunHansen2022, comp}. The SCI hierarchy is closely connected to S. Smale's \cite{Smale2, Smale_Acta_Numerica} program on the foundations of computational mathematics and scientific computing. This program motivated the early investigations of polynomial root-finding by C. McMullen \cite{McMullen1, McMullen2, Smale_McMullen} and P. Doyle \& C. McMullen \cite{Doyle_McMullen}. Their work yielded pioneering classification results within the SCI hierarchy, predating its formal introduction. See also the classification results in the SCI hierarchy by S. Weinberger \cite{Weinberger}. Although some of the mathematics underlying the SCI hierarchy can be traced back to K. G\"odel \cite{godel1931formal} and A. Turing \cite{Turing_Machine}, the newly developed techniques apply to any model of computation.

\subsubsection{An informal introduction to the SCI hierarchy}

Here we informally introduce the general SCI hierarchy (see \cite{SCI, ColbrookAntunHansen2022} for the formal definitions). The SCI hierarchy is designed to accommodate all types of computational problems and categorises the difficulty of computational problems according to the number of limits required to solve them. A computational problem is given by: a set of inputs $\Omega$ (where each element of $\Omega$ specifies a problem instance), a metric space $(\mathcal{M}, d)$ of candidate solutions, and a solution map $\Xi : \Omega \rightrightarrows \mathcal{M}$ that takes each problem instance to its solution(s) (the double arrow indicates that the solution could be multi-valued). There is in general also an evaluation set $\Lambda$ specifying which information about the problem instances may actually be accessed by an algorithm. This will not be relevant \footnote{For the computational problems discussed in this appendix, each problem instance in $\Omega$ is specified entirely by a single natural number, so $\Lambda$ amounts to containing only an identity map.} for this appendix. A computational problem is often specified simply as $\{\Xi, \Omega\}$, when $\mathcal{M}$ and $\Lambda$ are understood. 

\emph{The SCI of a computational problem is defined as the minimal number of limits \footnote{This is formalised by the notion of towers of algorithms.} required to solve that problem ($+\infty$ if no finite number of limits suffices)}. The mainstay classes of the hierarchy are the $\Delta_{m}^{\alpha}$ classes, which are defined for each $m \geq 0$. Here $\alpha$ is related to the model of computation, as explained below. Informally, we have the following description. Given a collection $\mathcal{C}$ of computational problems:
\begin{itemize}
\item[(i)] $\Delta^{\alpha}_0$ is the set of problems that can be computed in finite time, so that they satisfy SCI $=0$;
\item[(ii)] $\Delta^{\alpha}_1$ is the set of problems that can be computed using one limit (they satisfy SCI $\leq1$) with control of the error, i.e., there is a sequence of algorithms $\{\Gamma_n\}$ with $d(\Gamma_n(A), \Xi(A)) \leq 2^{-n}, \, \forall A \in \Omega$;
\item[(iii)] $\Delta^{\alpha}_2$ is the set of problems that can be computed using one limit (again SCI $\leq 1$) but sans error control, i.e., there is a sequence of algorithms $\{\Gamma_n\}$ with $\lim_{n\rightarrow \infty}\Gamma_n(A) = \Xi(A), \, \forall A \in \Omega$;
\item[(iv)] $\Delta^{\alpha}_{m+1}$, for $m \geq 1$, is the set of problems that can be computed by using $m$ limits, (in other words, SCI $\leq m$), i.e., there is a family of algorithms $\{\Gamma_{n_m, \hdots, n_1}\}$ with 
$$
\lim_{n_m \rightarrow\infty}\hdots \lim_{n_1\rightarrow\infty}\Gamma_{n_m,\hdots, n_1}(A) = \Xi(A), \, \forall A \in \Omega.
$$
\end{itemize}
In general, this hierarchy can only be refined further, beyond these classes, if the metric space $\mathcal{M}$ admits additional structure, such as being totally ordered. On the other hand, the hierarchy typically does not collapse, so that we have:
\begin{equation}\label{SCI1}
 \Delta_0^{\alpha} \subsetneq \Delta_1^{\alpha} \subsetneq \Delta_2^{\alpha} \subsetneq \hdots \subsetneq \Delta^{\alpha}_{m} \subsetneq \hdots.
\end{equation}
Depending on the collection $\mathcal{C}$ of computational problems, however, the hierarchy \eqref{SCI1} may terminate for a finite $m$ or continue for arbitrary large $m$. 

\begin{definition}[The SCI Hierarchy (totally ordered set)]\label{def:tot_ord}
Given the setup above, suppose in addition that $\mathcal{M}$ is a totally ordered set. 
Define 
\begin{equation*}
\begin{split}
\Sigma^{\alpha}_0 &= \Pi^{\alpha}_0 = \Delta^{\alpha}_0,\\
\Sigma^{\alpha}_{1} &= \{\{\Xi,\Omega\} \in \Delta_{2}^{\alpha} \ \vert \  \exists \ \textrm{ a sequence of algorithms } \{\Gamma_{n}\} \text{ s.t. } \Gamma_{n}(A) \nearrow \Xi(A) \ \, \forall A \in \Omega\}, 
\\
\Pi^{\alpha}_{1} &= \{\{\Xi,\Omega\} \in \Delta_{2}^{\alpha} \ \vert \  \exists \ \textrm{ a sequence of algorithms } \{\Gamma_{n}\} \text{ s.t. } \Gamma_{n}(A) \searrow \Xi(A) \ \, \forall A \in \Omega\},
\end{split}
\end{equation*}
where $\nearrow$ and $\searrow$ denotes convergence from below and above respectively, as well as, for $m \in \mathbb{N}$, 
\begin{equation*}
\begin{split}
\Sigma^{\alpha}_{m+1} &= \{\{\Xi,\Omega\} \in \Delta_{m+2}^{\alpha} \ \vert \  \exists \ \textrm{ height } m +1 \textrm{ tower } \{\Gamma_{n_{m+1}, \hdots, n_1}\} \textrm{ for } \{ \Xi, \Omega \} \text{ s.t. }\Gamma_{n_{m+1}}(A) \nearrow \Xi(A) \ \, \forall A \in \Omega\}, \\
\Pi^{\alpha}_{m+1} &= \{\{\Xi,\Omega\} \in \Delta_{m+2}^{\alpha} \ \vert \  \exists \ \textrm{ height } m +1 \textrm{ tower } \{\Gamma_{n_{m+1}, \hdots, n_1}\} \textrm{ for } \{ \Xi, \Omega \} \text{ s.t. }\Gamma_{n_{m+1}}(A) \searrow \Xi(A) \ \, \forall A \in \Omega\}.
\end{split}
\end{equation*}
\end{definition}
See \cite{SCI} for the formal definition of the towers of algorithms used to define $\Sigma_{m+1}^{\alpha}$ and $\Pi_{m+1}^{\alpha}$ above. A specific case of the definitions of such towers, when the algorithms are Turing machines, is elaborated upon below.
The above definitions apply in particular to the totally ordered metric space $\mathcal{M} = \{0,1\}$; this yields the SCI hierarchy for arbitrary decision problems. 
Schematically, the general SCI hierarchy can be viewed in the following way:

\begin{equation}\label{SCI_hierarchy}
\begin{tikzpicture}[baseline=(current  bounding  box.center)]
  \matrix (m) [matrix of math nodes,row sep=1.2em,column sep=1.5em] {
  \Pi_0^{\alpha}   &                    & \Pi_1^{\alpha} &    &  \Pi_2^{\alpha}&  & {}\\
  \Delta_0^{\alpha}&  \Delta_1^{\alpha} & \Sigma_1^{\alpha}\cup\Pi_1^{\alpha} & \Delta_2^{\alpha}&      \Sigma_2^{\alpha}\cup\Pi_2^{\alpha} & \Delta_3^{\alpha}& \cdots\\
	\Sigma_0^{\alpha}&                    & \Sigma_1^{\alpha} & &  \Sigma_2^{\alpha}&  &{} \\
  };
 \path[-stealth, auto] (m-1-1) edge[draw=none]
                                    node [sloped, auto=false,
                                     allow upside down] {$=$} (m-2-1)
																		(m-3-1) edge[draw=none]
                                    node [sloped, auto=false,
                                     allow upside down] {$=$} (m-2-1)
																		
																		(m-2-2) edge[draw=none]
                                    node [sloped, auto=false,
                                     allow upside down] {$\subsetneq$} (m-2-3)
																		(m-2-3) edge[draw=none]
                                    node [sloped, auto=false,
                                     allow upside down] {$\subsetneq$} (m-2-4)
																		(m-2-4) edge[draw=none]
                                    node [sloped, auto=false,
                                     allow upside down] {$\subsetneq$} (m-2-5)
																		(m-2-5) edge[draw=none]
                                    node [sloped, auto=false,
                                     allow upside down] {$\subsetneq$} (m-2-6)
																		(m-2-6) edge[draw=none]
                                    node [sloped, auto=false,
                                     allow upside down] {$\subsetneq$} (m-2-7)

												(m-2-1) edge[draw=none]
                                    node [sloped, auto=false,
                                     allow upside down] {$\subsetneq$} (m-2-2)
											 (m-2-2) edge[draw=none]
                                    node [sloped, auto=false,
                                     allow upside down] {$\subsetneq$} (m-1-3)
											(m-2-2) edge[draw=none]
                                    node [sloped, auto=false,
                                     allow upside down] {$\subsetneq$} (m-3-3)
											 (m-1-3) edge[draw=none]
                                    node [sloped, auto=false,
                                     allow upside down] {$\subsetneq$} (m-2-4)
																		(m-3-3) edge[draw=none]
                                    node [sloped, auto=false,
                                     allow upside down] {$\subsetneq$} (m-2-4)
																		(m-2-4) edge[draw=none]
                                    node [sloped, auto=false,
                                     allow upside down] {$\subsetneq$} (m-1-5)
											(m-2-4) edge[draw=none]
                                    node [sloped, auto=false,
                                     allow upside down] {$\subsetneq$} (m-3-5)
																		(m-1-5) edge[draw=none]
                                    node [sloped, auto=false,
                                     allow upside down] {$\subsetneq$} (m-2-6)
																		(m-3-5) edge[draw=none]
                                    node [sloped, auto=false,
                                     allow upside down] {$\subsetneq$} (m-2-6)
											(m-2-6) edge[draw=none]
                                    node [sloped, auto=false,
                                     allow upside down] {$\subsetneq$} (m-1-7)
																		(m-2-6) edge[draw=none]
                                    node [sloped, auto=false,
                                     allow upside down] {$\subsetneq$} (m-3-7);
																		
\end{tikzpicture}
\end{equation}
Note that the $\Sigma_1^{\alpha}$ and $\Pi_1^{\alpha}$ classes are crucial as they guarantee the existence of algorithms that will never make mistakes. They thus become crucial in computer-assisted proofs, see \cite[Section~3]{SCI}.

We also use standard facts about the arithmetical hierarchy and oracle computability, for which we refer the reader to \cite[Chapter~4]{Soare2016}.
The $\alpha$ in the $\Delta_{m}^{\alpha}$ classes can refer to any model of computation (e.g. the classical Turing model \cite{Turing_Machine} or the Blum-Shub-Smale model \cite{BSS89}). Hereon, we will take $\alpha = A$ where $A$ refers to the arithmetic model of computation. This means the algorithms are all Turing machines i.e. the classical Turing model.

\begin{remark}[SCI hierarchy generalises the arithmetical hierarchy]
We note that the arithmetical hierarchy becomes a special case of the SCI hierarchy \cite[Proposition~A.8]{SCI}. More precisely, with $\alpha = A$ and $\mathcal{M} = \{0, 1\}$, if we restrict to computational problems asking about membership in a subset of $\mathbb{N}$, we have $\Sigma_{m}^{A} = \Sigma_{m}^{0}$, $\Pi_{m}^{A} = \Pi_{m}^{0}$ and $\Delta_{m}^{A} = \Delta_{m}^{0}$ for every $m \geq 0$ \cite[Proposition~A.8]{SCI}. 
Here, $\Sigma_{m}^{0}$,  $\Pi_{m}^{0}$ and $\Delta_{m}^{0}$ are the classical classification classes from the arithmetical hierarchy.   
We will make use of this fact in the following results.
\end{remark}

\begin{notation}[Reducibility]
For sets $A,B\subseteq\mathbb N$, we write $A \leq_m B$ if there is a total computable function $f:\mathbb N\to\mathbb N$ such that
\[
(\forall n \in \mathbb{N}) [ n\in A\iff f(n)\in B].
\]
In this case we say $A$ is many-one reducible to $B$, witnessed by $f$.
We write $A \leq_T B$ if there is an oracle Turing machine which, with $B$ on the oracle tape, computes the characteristic function of $A$. In this case we say $A$ is Turing reducible to $B$.
The Turing degree of a set $A \subseteq \mathbb{N}$ is the set of all $C \subseteq \mathbb{N}$ that simultaneously satisfy $A \leq_T C$ and $C \leq_T A$.
We say $A$ is $\Sigma_{n}^{0}$-complete if $A \in \Sigma_{n}^{0}$ and also $C \leq_m A$ for every $C \in \Sigma_{n}^{0}$. Similarly, $A$ is $\Pi_{n}^{0}$-complete if $A \in \Pi_{n}^{0}$ and $C \leq_m A$ for every $C \in \Pi_{n}^{0}$.
We fix an integer-coefficient polynomial pairing function $\langle \textrm{ } , \textrm{ } \rangle : \mathbb{N}^2 \to \mathbb{N}$ given by
\begin{equation*}
\langle x, y \rangle = (x + y)(x + y + 1) + 2y, \textrm{ for  } (x, y) \in \mathbb{N}^2,
\end{equation*}
which is a computable injection, by \cite[Section~1, p.~1001]{Nathanson2016}. By iteration, for each $n \geq 2$ we also obtain and hereon fix an $n$-ary pairing function $\langle \textrm{ }, ..., \textrm{ } \rangle : \mathbb{N}^n \to \mathbb{N}$, which is also an integer-coefficient polynomial.
\end{notation}

\subsection{Resolving ambiguities is hard}
Recall the condition defining $n_e$ in Example~\ref{ex:example3}.
For an enumeration $(p_e)_{e\in\mathbb N}$ of integer-coefficient
polynomials in the variables $n,x_1,\ldots,x_k$, write
\begin{equation}\label{eq:presupposition-base-predicate}
    B_e(n)\iff
    (\forall(x_1,\ldots,x_k)\in\mathbb N^k)\textrm{ }
    p_e(n,x_1,\ldots,x_k)\ne0.
\end{equation}
Then $n_e$ is defined to be the least $n\in\mathbb N$ for which $B_e(n)$ holds, provided such an $n$ exists, and is undefined otherwise. Completeness throughout this subsection refers to completeness under many-one reductions.

\begin{theorem}\label{thm:presupposition}
Fix $k\geq9$. For any computable enumeration $(p_e)_{e\in\mathbb N}$ of all
polynomials with integer coefficients in $k+1$ variables, let $B_e$
be as in~\eqref{eq:presupposition-base-predicate}. Then the set
\[
    d:=\{e\in\mathbb N:(\exists n\in\mathbb N)\textrm{ }B_e(n)\}
\]
is $\Sigma^0_2$-complete. In particular, no AI can always determine
whether the number $n_e$ specified in Example~\ref{ex:example3} is well-defined, even with
an oracle for the Halting problem.
\end{theorem}
\begin{proof}
We use the usual $\mu$ notation for minimum search \cite{Soare2016}. By the assumption that the enumeration $e \mapsto p_e$ is computable and the assumption that $p_e$ has integer coefficients, there exists an algorithm that decides if $p_e(n,x_1, \ldots x_k) = 0$ given inputs $n,x_1,\dotsc,x_k$. The condition defining membership in $d$ is thus given by the quantifiers $(\exists n \in \mathbb{N})(\forall x_1 \in \mathbb{N}) ... (\forall x_k \in \mathbb{N})$ over a decidable relation, so we have $d\in\Sigma^0_2$.

To obtain $\Sigma^{0}_{2}$-completeness, fix a standard effective enumeration $(P_e)_{e\in\mathbb N}$ of all Turing machines. For $e, n \in \mathbb{N}$, write $P_e(n) \downarrow$ when $P_e(n)$ halts and $P_e(n) \uparrow$ when $P_e(n)$ does not halt. Recall that a Diophantine set~\cite{Matiyasevich1993} is a set $A \subseteq \mathbb{N}$ that can be represented using a finite number of unknowns, in the sense that there exists an integer $s \geq 1$ and an integer-coefficient polynomial $p(n, x_1, ..., x_s)$ in $s+1$ variables with:
\begin{equation*}
n \in A \iff (\exists (x_1, \ldots, x_s) \in \mathbb{N}^s) \textrm{ }p(n, x_1, \ldots, x_s) = 0.
\end{equation*}
The Davis--Putnam--Robinson--Matiyasevich (DPRM) theorem~\cite{Matiyasevich1993} states that the Diophantine sets are precisely the computably enumerable sets. The nine-unknown refinement~\cite{Jones1982} of the DPRM theorem permits every computably enumerable relation over $\mathbb N$ to be represented using nine unknowns.
Applying this to the computably enumerable set $K_1 = \{ \langle e, n \rangle \in \mathbb{N} : P_e(n) \downarrow \}$, we find there exists a fixed integer-coefficient polynomial $q_0(\tilde{n}, x_1, ..., x_9)$, where $\tilde{n}$ is a variable, such that $\langle e, n \rangle \in K_1$ if and only if $q_0(\langle e, n \rangle, x_1, ..., x_9)$ admits a root $(x_1, ..., x_9) \in \mathbb{N}^{9}$. Defining the polynomial $q(e, n, x_1, ..., x_9) := q_0(\langle e, n \rangle, x_1, ..., x_9)$, we have for each $e, n \in \mathbb{N}$:
\[
    P_e(n)\downarrow
    \quad\Longleftrightarrow\quad
    (\exists (x_1, \ldots, x_9) \in\mathbb N^9)\textrm{ } q(e,n,x_1,\ldots, x_9)=0.
\]
For $k>9$, regard $q$ as independent of the additional variables
$x_{10},\ldots,x_k$. Define the binary predicate $W \subseteq \mathbb{N}^2$ by
\[
W(e, j) \iff p_j(n, x_1, ..., x_k) \equiv q(e, n, x_1, ..., x_k), 
\]
where $\equiv$ denotes that the two polynomials, whose variables are $n, x_1, ..., x_k$, have coinciding coefficients; given $e$ and $j$ this is decidable. Since $(p_e)_{e \in \mathbb{N}}$ enumerates all integer-coefficient polynomials in $k+1$ variables, the function $t(e) := \mu j[W(e, j)]$ is well-defined and computable. By the definitions of $d$ and $q$, we have
\[
    t(e)\in d
    \quad\Longleftrightarrow\quad
    (\exists n\in\mathbb N)\textrm{ }P_e(n)\uparrow
\]
In other words, $t$ is a computable many-one reduction from
\[
    \mathrm{NTOT}:=\{e\in\mathbb N:(\exists n\in\mathbb N)\textrm{ }P_e(n)\uparrow\}
\]
to $d$. Since $\mathrm{NTOT}$ is $\Sigma^0_2$-complete, so is $d$. By Post's theorem \cite[Chapter~4]{Soare2016}, we have $d \leq_{T} \emptyset' \iff d \in \Delta_{2}^{0}$. Since $d \in \Sigma_{2}^{0}$, we have $d \leq_T \emptyset' \iff d \in \Pi_{2}^{0}$. But whenever a set belongs to $\Pi_{2}^{0}$ so too does any set that (many-one) reduces to it; thus if $d \in \Pi_{2}^{0}$ then by $\Sigma_{2}^{0}$-completeness we would have $\Sigma_{2}^{0} \subseteq \Pi_{2}^{0}$, contradicting strictness of the arithmetical hierarchy \cite[Corollary~4.2.3, p.~85]{Soare2016}. Consequently, we cannot have $d \leq_{T} \emptyset'$. Since $d$ is exactly the set of
$e$ for which $n_e$ in Example~\ref{ex:example3} is well-defined, the final claim follows.
\end{proof}

The preceding result can be extended to arbitrarily high levels in the SCI hierarchy whilst simultaneously retaining polynomial predicates and the least-number construction of Example~\ref{ex:example3}. We first generalise the predicate $B_e(n)$.
For $l\geq2$, consider an enumeration $(p_e)_{e\in\mathbb N}$ of
integer-coefficient polynomials in the $l-1+k$ variables
\[
    n,z_2,\ldots,z_{l-1},x_1,\ldots,x_k.
\]
Let $Q_j$ denote $\exists$ for odd $j$ and $\forall$ for even $j$. We omit $z_2,\ldots,z_{l-1}$ when $l=2$. Define $B^{(l)}_e(n)$ by the following two cases.
\begin{subequations}\label{eq:presupposition-level-predicate}
For even $l$, again omitting $z_2, \ldots z_{l-1}$ when $l = 2$, we have
\begin{equation}
    \begin{aligned}
        B^{(l)}_e(n)\iff
        &(Q_2z_2 \in \mathbb{N})\;\cdots\;(Q_{l-1}z_{l-1} \in \mathbb{N})\;
        (\forall(x_1,\ldots,x_k)\in\mathbb N^k)\\
        &p_e(n,z_2,\ldots,z_{l-1},x_1,\ldots,x_k)\ne0,
    \end{aligned}
\end{equation}
For odd $l \geq 3$,
\begin{equation}
    \begin{aligned}
        B^{(l)}_e(n)\iff
        &(Q_2z_2 \in \mathbb{N})\;\cdots\;(Q_{l-1}z_{l-1}\in\mathbb{N})\;
        (\exists(x_1,\ldots,x_k)\in\mathbb N^k)\\
        &p_e(n,z_2,\ldots,z_{l-1},x_1,\ldots,x_k)=0.
    \end{aligned}
\end{equation}
\end{subequations}
Note that (using the same $k$ and polynomial enumeration) $B^{(2)}_e(n)$ is $B_e(n)$. As in \cite{Soare2016}, we write $\emptyset^{(0)} = \emptyset$ and $\emptyset^{(j+1)}$ for the Halting problem relative to $\emptyset^{(j)}$, for $j \in \mathbb{N}$.

\begin{theorem}[Presupposition resolution at arbitrary arithmetical levels]
\label{thm:presupposition-arbitrary-levels}
Fix $l\geq2$ and $k\geq9$. Fix a computable enumeration
$(p_e)_{e\in\mathbb N}$ of all integer-coefficient polynomials in $l-1+k$ variables, and let $B^{(l)}_e$ be as
in~\eqref{eq:presupposition-level-predicate}. Then
\[
d_l:=\{e\in\mathbb N:(\exists n\in\mathbb N)\textrm{ }B^{(l)}_e(n)\}
\]
is $\Sigma_{l}^{0}$-complete. In particular, no algorithm can always determine whether the least $n$ for which $B^{(l)}_e(n)$ holds is well-defined, even with an oracle for $\emptyset^{(l-1)}$.
\end{theorem}

\begin{proof}
Fix $l\geq2$. By its definition, $B^{(l)}_e(n)$ is a $\Pi^0_{l-1}$ predicate. Indeed, it begins with $l-2$ alternating quantifiers - the first being a universal quantifier -- and the subsequent $k$ quantifiers are all either existential or all universal; the affected predicate therein, namely $p_e(n, z_2, ..., z_{l-1} , x_1, ..., x_k) = 0$, is decidable because the polynomial enumeration is computable and the computation is over the integers. That is, $d_l \in\Sigma_{l}^{0}$.
For $\Sigma_{l}^{0}$-hardness, let $A\subseteq\mathbb N$ be any $\Sigma_{l}^{0}$ set. By definition, there exists a computable relation $R \subseteq \mathbb{N}^{l+1}$ such that
\begin{equation}\label{eq:presupposition-level-normal-form}
 a \in A \iff   (Q_1z_1 \in \mathbb{N})\;\cdots\;(Q_lz_l \in \mathbb{N})\textrm{ }R(a,z_1,\ldots,z_l).
\end{equation}
Define a computably enumerable relation on the remaining parameters by
\[
    C(a,z_1,\ldots,z_{l-1})\quad\Longleftrightarrow\quad
    \begin{cases}
        (\exists z_l \in \mathbb{N})\textrm{ }R(a,z_1,\ldots,z_l), & \text{ for } l\text{ odd};\\
        (\exists z_l \in \mathbb{N})\textrm{ }\neg R(a,z_1,\ldots,z_l),  &\textrm{ for } l\text{ even}.
    \end{cases}
\]
By the nine-unknown refinement of the DPRM theorem \cite{Matiyasevich1993,Jones1982} applied to the computably enumerable set $\{ \langle a, z_1, ..., z_{l-1} \rangle \in \mathbb{N} : C(a, z_1, ..., z_{l-1})\}$, there exists a fixed integer-coefficient polynomial $q_0(\tilde{n}, x_1, ..., x_9)$, where $\tilde{n}$ is a variable, such that
\begin{equation*}
C(a, z_1, ..., z_{l-1}) \iff (\exists (x_1, ..., x_9) \in \mathbb{N}^9) \textrm{ }q_0(\langle a, z_1, ..., z_{l-1} \rangle, x_1, ...., x_9) = 0.
\end{equation*}
Defining the polynomial $q(a, z_1, ..., z_{l-1}, x_1, ..., x_9) := q_0(\langle a, z_1, ..., z_{l-1} \rangle, x_1, ...., x_9)$, we have
\[
    C(a,z_1,\ldots,z_{l-1})\iff
    (\exists(x_1,\ldots,x_9)\in\mathbb N^9)\textrm{ }
    q(a,z_1,\ldots,z_{l-1},x_1,\ldots,x_9)=0.
\]
If $l$ is odd, this representation permits \eqref{eq:presupposition-level-normal-form} to be expressed as
\begin{equation*}
\begin{aligned}
a \in A \iff & (\exists z_1 \in \mathbb{N})\;\cdots\;(\forall z_{l-1} \in \mathbb{N}) (\exists (x_1, ..., x_9) \in \mathbb{N}^9) \\ & q(a, z_1, ..., z_{l-1}, x_1, ..., x_9) = 0.
\end{aligned}
 \end{equation*}
If $l$ is even, negation of the defining condition for $C$ gives
\[
    (\forall z_l \in \mathbb{N})\textrm{ }R(a,z_1,\ldots,z_l)
    \iff
    (\forall(x_1,\ldots,x_9)\in\mathbb N^9)\textrm{ }q(a,z_1,\ldots,z_{l-1},x_1,\ldots,x_9)\ne0,
\]
so that
\eqref{eq:presupposition-level-normal-form} can be expressed as
\begin{equation*}
\begin{aligned}
a \in A \iff  & (\exists z_1 \in \mathbb{N})(\forall z_2 \in \mathbb{N})\;\cdots\;(\exists z_{l-1} \in \mathbb{N}) (\forall (x_1, ..., x_9) \in \mathbb{N}^9) \\ & q(a, z_1, ..., z_{l-1}, x_1, ..., x_9) \neq 0.
\end{aligned}
\end{equation*}
Consequently, regardless of the parity of $l$, we have replaced the final quantification $(Q_l z_l \in \mathbb{N})$ in \eqref{eq:presupposition-level-normal-form} sans appending an additional quantifier alternation. When $k>9$, we regard $q$ as independent of the additional variables $x_{10},\ldots,x_k$.
Define the binary predicate $W \subseteq \mathbb{N}^2$ by
\begin{equation*}
W(a, j) \iff p_j(n, z_2, ..., z_{l-1}, x_1, ..., x_k) \equiv q(a, n, z_2 ..., z_{l-1}, x_{1}, ..., x_k),   
\end{equation*}
where $\equiv$ denotes that the two polynomials, whose variables are $n, z_2, ..., z_{l-1}, x_1, ..., x_k$, have coinciding coefficients; given $a$ and $j$ this is decidable.
Since $(p_e)_{e \in \mathbb{N}}$ enumerates all integer-coefficient polynomials in $l-1+k$ variables, the function $t(a) := \mu j [W(a, j)]$ is well-defined and computable.
The preceding equivalences imply that
\[
    a\in A\iff t(a)\in d_l.
\]
That is, $A\leq_m d_l$, establishing $\Sigma^0_l$-completeness of $d_l$.
By definition of $d_l$, the least $n \in \mathbb{N}$ satisfying $B^{(l)}_e(n)$ exists exactly when $e\in d_l$.
By Post's theorem \cite{Soare2016}, we have $d_l \leq_T \emptyset^{(l-1)} \iff d_l \in \Delta_{l}^{0}$. Since $d_l \in \Sigma_{l}^{0}$, this becomes $d_l \leq_T \emptyset^{(l-1)} \iff d_l \in \Pi_{l}^{0}$. But whenever a set belongs to $\Pi_{l}^{0}$ so too does any set that (many-one) reduces to it; thus if $d _l\in \Pi_{l}^{0}$, then by $\Sigma_{l}^{0}$-completeness we would have $\Sigma_{l}^{0} \subseteq \Pi_{l}^{0}$, contrary to the strictness of the arithmetical hierarchy. Consequently, we cannot have $d_l \leq_T \emptyset^{(l-1)}$.
\end{proof}

When we use the same $k$ and the same polynomial enumeration,  the case $l=2$ gives $B_{e}^{(2)}(n) = B_e(n)$ for every $n \in \mathbb{N}$, so that $d_2 = d$, which recovers Theorem~\ref{thm:presupposition}.
Moreover, for every $l\geq2$, the complement of $d_l$ is
$\Pi_{l}^{0}$-complete, and $d_l$ has the Turing degree of $\emptyset^{(l)}$ (since $\emptyset^{(l)}$ is $\Sigma_{l}^{0}$-complete by Post's Theorem \cite{Soare2016}).
For each fixed $l\geq2$ and $e \in \mathbb{N}$, let $n^{(l)}_e$ be the least $n$ for which $B^{(l)}_e(n)$ holds, and define $r^{(l)}_e =1/(n^{(l)}_e+1)$.
Both $n^{(l)}_e$ and $r^{(l)}_e$ are left undefined if there is no such $n$.
Replacing $n_e$ and $r_e$ in Example~\ref{ex:example3} by $n^{(l)}_e$ and $r^{(l)}_e$, respectively, gives the same
least-number construction. The displayed identity
$(r^{(l)}_e+1)^2=(r^{(l)}_e)^2+2r^{(l)}_e+1$ remains elementary to establish, while as $e$ varies, deciding whether the stated definition of $r^{(l)}_e$ refers to a well-defined number is a $\Sigma^0_l$-complete problem. Accordingly, we have established the existence of presupposition-resolution problems of
unbounded finite arithmetical complexity.

For the SCI classification of the presupposition resolution problem we have thus far discussed, we use \emph{arithmetic (computable) towers}; see \cite[Definitions~5.2--5.6]{colbrook2019foundations} for the formal definition.
The input space $\Omega$ is $\mathbb N$ and the output space is $\mathcal{M} = \{0,1\}$ with the discrete metric. A tower of height $h \geq 1$ for a function $\Xi : \Omega \to \{0, 1\}$ is given by a bottom-level computable function $\Gamma : \Omega \times \mathbb{N}^h \to \{0, 1\}$ that satisfies
\[
    \Xi(e)=\lim_{s_1\to\infty}\cdots\lim_{s_h\to\infty}
    \Gamma(e,s_1,\ldots,s_h),
\]
for each $e \in \Omega$. In particular, each successive limit exists (for every fixed choice of the outer indices). An arithmetical tower of height zero for $\Xi$ is simply a computable function $\Gamma : \Omega \to \{0,1\}$ with $\Gamma = \Xi$. Given any function $\Xi : \Omega \to \{0, 1\}$, we define $\SCI_{\mathrm A}(\Xi)$ to be the least $h \geq 0$ for which there is an arithmetical tower of height $h$ for $\Xi$; we define $\SCI_{\mathrm A}(\Xi) = +\infty$ when no such tower exists. We also define $\Delta^{\mathrm A}_{h+1}$ to denote the class of problems $\Xi$ with $\SCI_{\mathrm A}(\Xi) \leq h$, for each $h \geq 1$.

\begin{corollary}[Unbounded SCI for presupposition resolution]
\label{cor:presupposition-sci}
For each $l\geq2$, define $\Xi_{d_l} : \Omega \to \{0, 1\}$ by $\Xi_{d_l}(e) = 1$ if and only if $e \in d_l$. With $k\geq9$ and the enumeration in Theorem~\ref{thm:presupposition-arbitrary-levels}, we have
\[
    \SCI_{\mathrm A}(\Xi_{d_l})=l
    \qquad\text{for every }l\geq2.
\]
Equivalently, the decision problem for the set $d_l$ lies in
$\Delta^{\mathrm A}_{l+1}\setminus\Delta^{\mathrm A}_l$.
\end{corollary}

\begin{proof}
Fix $l \geq 2$. Since $d_l \in \Sigma_{l}^{0}$, we have $d_l \in \Sigma_{l}^{\mathrm A}$ (as the SCI hierarchy generalises the arithmetical hierarchy, as discussed above). This means $\SCI_{\mathrm A}(\Xi_{d_l}) \leq l$.

On the other hand, if $\SCI_{\mathrm A}(\Xi_{d_l}) < l$, then there is an arithmetical tower of height $0 \leq h < l$ for $\Xi_{d_l}$. This would give that $d_l \in \Delta_{h+1}^{0}$ (again using that the SCI hierarchy generalises the arithmetical hierarchy) and thus also $d_l \in \Delta_{l}^{0}$ since $h < l$. But $d_l$ is $\Sigma_{l}^{0}$-complete so cannot lie in $\Delta_{l}^{0}$ ( else $\Sigma_{l}^{0} \subseteq \Pi_{l}^{0}$, contrary to strictness of the arithmetical hierarchy \cite[Corollary~4.2.3, p.~85]{Soare2016}).
This means $\SCI_{\mathrm{A}}(\Xi_{d_l}) \geq l$. We conclude that $\SCI_{\mathrm{A}}(\Xi_{d_l}) = l$, as required.
\end{proof}

\subsubsection{The problem of resolving ambiguities has SCI $= \infty$}\label{sec:ambSCI}
Recall Example \ref{ex:example3} and $B_e(n)$ from \eqref{eq:pred}. Recall that the ambiguity in Example \ref{ex:example3} -- whether $n_e$ is well defined or not -- was caused by the following question:
\begin{quote}
	`\textit{Is $\{n\in\mathbb{N}:B_e(n)\}$ empty or not}?
\end{quote}
Note that we could have made Example \ref{ex:example3}, and the definition of $n_e$, more general yielding an ambiguity caused by the question:
\begin{quote}
	`\textit{Is $\{n\in\mathbb{N}:B^{(l)}_e(n)\}$ empty or not}? 
\end{quote}
  Here, $B^{(l)}_e(n)$ is defined for integers $l \geq 2$ in \eqref{eq:presupposition-level-predicate}. Note that a trustworthy AI autoformaliser that will always provide an output: either a semantically faithful translation or `this statement cannot be translated' must resolve the above question for each $e \in \mathbb{N}$ and $l \geq 2$. However, by Corollary \ref{cor:presupposition-sci} -- for each fixed $l$ -- this problem has SCI $= l$. Hence, the problem of answering the above question for all $e \in \mathbb{N}$ and $l \geq 2$ has SCI $= \infty$. 
  We can now make sense of our informal statement:
 \begin{quote}
	\emph{
		Providing semantically faithful AI autoformalisation is harder than any computational problem. 
	}
\end{quote}
 By `semantically faithful AI autoformalisation' we mean a trustworthy AI autoformaliser that will always provide an output -- as above. Also, by `any computational problem' we mean any computational problem with finite SCI.  So our statement above -- which was made informal for the purpose of readability -- means: Every trustworthy semantically faithful AI formaliser (as above) must solve problems that are harder than any problem with finite SCI.
 
 \begin{remark}
	The requirement that $k \geq 9$ represents the current state of literature on the DPRM theorem; this is an active area of research with a recent refinement~\cite{sun2026} lowering the bound from $k\geq11$ to $k\geq10$ in the case where $x \in \mathbb{Z}^k$ (in our theorem and proof we work with $x \in \mathbb{N}^k$, where the best known bound for $\Sigma_2^0$ completeness of $d$ is $k\geq 9$).
\end{remark}

\bibliographystyle{abbrv}
\bibliography{references_3}


\end{document}